\documentclass[a4paper,11pt]{article}
\usepackage{fullpage}

\usepackage{amsthm}
\usepackage{amssymb}
\usepackage{amsbsy}
\usepackage{amsmath}
\usepackage{mathrsfs}
\usepackage{amstext}  
\usepackage[dvipsnames]{xcolor}
\usepackage{graphics}
\usepackage{graphicx}
\usepackage[colorlinks]{hyperref}
\usepackage{enumitem}
\usepackage{multicol}
\usepackage{upquote} 
\usepackage[utf8]{inputenc}
\usepackage[T1]{fontenc}
\usepackage{authblk}

\usepackage{caption}
\usepackage{subcaption}

\usepackage[normalem]{ulem}
\usepackage{changepage}
\usepackage{thm-restate}

\usepackage{longtable}

\newtheorem{theorem}{Theorem}[section]  
\newtheorem{proposition}[theorem]{Proposition} 
\newtheorem{lemma}[theorem]{Lemma} 
\newtheorem{corollary}[theorem]{Corollary}

\newtheorem{conjecture}[theorem]{Conjecture}
\newtheorem{question}[theorem]{Question}

\newcommand{\x}{\chi_{\star}}
\newcommand{\purple}{\operatorname{\textcolor{Purple}{purple}}}
\newcommand{\red}{\operatorname{\textcolor{red}{red}}}
\newcommand{\green}{\operatorname{\textcolor{Green}{green}}}
\newcommand{\orange}{\operatorname{\textcolor{Orange}{orange}}}
\newcommand{\blue}{\operatorname{\textcolor{blue}{blue}}}

\newcommand{\spn}{\mathcal{S}}

\date{}

\title{Star Coloring on Some Subclasses of Chordal Graphs}

\author[1]{Germán Benítez-Bobadilla\thanks{\texttt{german@ciencias.unam.mx}}}
\author[2]{Fernando Esteban Contreras-Mendoza\thanks{\texttt{esteban.contreras@im.unam.mx}}}
\author[1]{César Hernández-Cruz\thanks{\texttt{chc@ciencias.unam.mx}}}
\author[3]{Cláudia Linhares Sales\thanks{\texttt{linhares@dc.ufc.br}}}
\author[1]{Ana Trujillo-Negrete\thanks{\texttt{ltrujillo@ciencias.unam.mx}}}

\affil[1]{Facultad de Ciencias, Universidad Nacional Autónoma de México, Ciudad de México, Mexico}
\affil[2]{Instituto de Matemáticas, Universidad Nacional Autónoma de México, Ciudad de México, Mexico}
\affil[3]{Departamento de Computa\c{c}\~ao, Universidade Federal do Cear\'a, Fortaleza, Brazil}

\date{\today} 
\begin{document}

\maketitle

\begin{abstract}
    A star coloring of a graph $G$ is a proper coloring in which no path on four
    vertices is bicolored. The star chromatic number $\x(G)$ is the minimum
    number of colors in a star coloring of $G$. In this work we study star
    colorings from the perspective of forbidden induced subgraphs, focusing on
    three subclasses of chordal graphs. We provide both a structural
    characterization and a characterization in terms of forbidden induced
    subgraphs for star $3$-colorable chordal graphs; such characterizations
    yield a simple certifying recognition algorithm, running in time
    $O(|V|+|E|)$, for this class. We also characterize split graphs that are
    star $4$-colorable and star $5$-colorable in terms of (finitely many)
    forbidden induced subgraphs, again deriving linear-time certifying
    recognition algorithms. Finally, we study star colorings of $2$-trees and
    $2$-paths: we characterize the $2$-paths that are star $4$-colorable, prove
    that every $2$-path is star $5$-colorable, and exhibit a $2$-tree on $21$
    vertices with star chromatic number $6$ such that any proper induced
    subgrahp has star chromatic number $5$.
\end{abstract}


\section{Introduction}

We consider graphs with neither loops nor parallel edges.   We usually assume
that a graph $G$ has vertex set $V(G)$ and edge set $E(G)$; when there is no
ambiguity we only use $V$ and $E$.   We follow \cite{bondy2008} for basic
notation, nomenclature and definitions.

For a positive integer $k$, a \textit{(proper) vertex $k$-coloring} (or simply
\textit{$k$-coloring}) of a graph $G$ is a mapping $c \colon V \to \{0, \dots,
k-1\}$ such that $c(u) \neq c(v)$ whenever $uv\in E(G)$.   For each $i \in \{ 0,
\dots, k-1 \}$, the \textit{color class} $V_i$ of the coloring $c$ is defined to
be $V_i = c^{-1}(i)$; notice that $c$ is not necessarily surjective and thus
some color classes could be empty.   We sometimes describe a $k$-coloring by
listing its color classes $(V_0, \dots, V_{k-1})$.   Any $k$-coloring of a graph
$G$ is a \textit{vertex coloring} (or simply \textit{coloring}) of $G$. An
\textit{acyclic coloring} of $G$ is a proper coloring in which no cycle of $G$
is bicolored---equivalently, for every pair of colors, the subgraph induced by
the corresponding color classes is a forest. The \textit{acyclic chromatic
number} of $G$, denoted by $\chi_a(G)$, is the minimum number of colors in an
acyclic coloring of $G$.   A \textit{star coloring} of $G$ is a proper coloring
in which no path on four vertices (not necessarily induced) is bicolored, that
is, there is no path of length~$3$ whose vertices use only two colors.
Equivalently, for every pair of colors, the subgraph induced by the
corresponding color classes is a disjoint union of stars (a \textit{star
forest}). The \textit{star chromatic number} of $G$, denoted by
$\x(G)$\footnote{This parameter has been usually denoted by $\chi_s$, but we
propose to use $\star$ instead of $s$ because it is clear how to read it and it
frees the letter $s$ to be used on other parameters.}, is the smallest integer
$k$ for which $G$ admits a star coloring with $k$ colors. Clearly, every star
coloring is an acyclic coloring (any bicolored cycle contains a bicolored path
on four vertices), and therefore
\begin{equation}
    \label{eq:inequalitiesChain}
    \omega(G) \le \chi(G) \le \chi_a(G)\le \x(G)\quad\text{for every graph $G$.}
\end{equation}

For a positive integer $k$, a \textit{star $k$-coloring} of a graph $G$ is a
star coloring of $G$ that uses at most $k$ colors. We say that $G$ is
\textit{star $k$-colorable} if it admits a star $k$-coloring, that is, if
$\x(G)\le k$.

\paragraph{Previous work.}
Gr\"unbaum introduced acyclic and star colorings in \cite{grunbaumIJM14}, where
he proved that every planar graph admits an acyclic coloring with 9 colors (and,
as he observed, `hence a star coloring with 2304 colors') and conjectured that 5
colors are enough for any planar graph.  In \cite{borodinDM25}, Borodin proved
Gr\"unbaum's conjecture, obtaining as a corollary that every planar graph admits
a star coloring with 80 colors.   Notice that Gr\"unbaum did not propose a name
for star colorings, and in 1979 Borodin used the name \textit{fan chromatic
number} for our star chromatic number.   (Fertin, Raspaud, and Reed
\cite{fertinJGT47} seem to be the first to use the name star coloring for this
kind of coloring.) In \cite{nesetril2003}, Ne\v set\v ril and Ossona de Mendez
improved the upper bound for the number of colors in a star coloring of a planar
graph to $30$, and shortly after, Albertson, Chappell, Kierstead, K\"ungden and
Ramamurthi \cite{albertsonEJC11} further improved this upper bound to $20$
colors.   For general planar graphs, the upper bound of $20$ for the star
chromatic number has not been improved, but it is not known to be sharp.  Many
authors have worked with additional restrictions to obtain better results, for
example, by restricting the girth of a planar graph.

From the short discussion in the previous paragraph, it is not hard to imagine
that most of the existing work related to star coloring of graphs is focused on
planar graphs.   Fertin, Raspaud and Reed \cite{fertinJGT47} obtained the star
chromatic number for some classic families of graphs (trees, cycles, complete
bipartite graphs, etc.) and calculated bounds for some other (grids, toroidal
grids, graphs of bounded treewidth, etc.).

Regarding complexity, in \cite{albertsonEJC11} it was proved that the problem of
deciding whether a planar bipartite graph has a star coloring with at most $3$
colors is NP-complete, and also that for any integers $k$ and $t$ such that $2
\le t \le k$, given a graph $G$ with $\chi(G) = t$, it is NP-complete to decide
whether $\x(G) \le k$.   Bok, Jedli\v ckov\'a, Martin, Ochem, Paulusma and Smith
\cite{bokJCSS154} studied the complexity of the problem of finding a star
coloring with at most $k$ colors (for $k \ge 3$) for a graph $G$ when $G$ is
$H$-free; they proved that the problem is solvable in polynomial time whenever
$H$ is a linear forest, and NP-complete otherwise. Additionally, when $k$ is
part of the input, the problem of determining whether an $H$-free graph $G$
admits a star coloring with at most $k$ colors is solvable in polynomial time
when $H$ is an induced subgraph of $P_4$ and NP-complete otherwise, except for
the open case when $H$ is isomorphic to $2K_2$.

Pursuing a question analogous to $\chi$-boundedness, for a variety of graph
classes $\mathcal{H}$, Karthick \cite{karthickGC34} provided upper bounds for
the star chromatic number of any graph $G$ in $\mathcal{H}$, in terms of its
clique number.

\paragraph{Additional definitions.}
A graph class $\mathcal{C}$ is called \textit{hereditary} if it is closed under
taking induced subgraphs; that is, whenever $G \in \mathcal{C}$ and $H$ is an
induced subgraph of $G$, then $H \in \mathcal{C}$.    Let $\mathcal{C}$ be a
hereditary class of graphs. A graph $F \notin \mathcal{C}$ is a \textit{minimal
obstruction to $\mathcal{C}$} if every proper induced subgraph of $F$ belongs
to $\mathcal{C}$. In other words, $F$ is minimal (with respect to the induced
subgraph relation) among all graphs that do not belong to $\mathcal{C}$.   The
\textit{obstruction set for $\mathcal{C}$} is the set $\mathcal{F}$ of all
minimal obstructions to $\mathcal{C}$. Thus, a graph $G$ belongs to
$\mathcal{C}$ if and only if it does not contain any member of $\mathcal{F}$ as
an induced subgraph; in such a case, we say that $\mathcal{C}$ \textit{admits a
characterization by minimal obstructions}.   Observe that every hereditary class
admits a characterization by minimal obstructions.

Given a family of graphs $\mathcal{C}$, an algorithm to recognize members of
$\mathcal{C}$ is \textit{certifying} if besides returning a boolean (yes or no)
answer, it also returns \textit{yes-certificates} and \textit{no-certificates}
for the user to verify that the answer provided by the algorithm is correct. For
example, a certifying algorithm for the recognition of bipartite graphs takes a
graph $G$ as an input and returns a bipartition of $G$ (a yes-certificate) if it
is bipartite or an odd cycle (a no-certificate) when it is not.

It is easy to see that, for every fixed integer $k$, the class of star
$k$-colorable graphs is hereditary, since the restriction of a star $k$-coloring
to an induced subgraph is again a star $k$-coloring.

If $G$ and $H$ are disjoint graphs, the \textit{join} of $G$ and $H$, denoted $G
\oplus H$, is the graph obtained from the disjoint union of $G$ and $H$ by
adding all edges joining a vertex of $G$ with a vertex of $H$.

\paragraph{Our contribution.}
The focus of this work is to characterize star $k$-colorable graphs in terms of
minimal obstructions when restricted to some well-known subfamilies of chordal
graphs.   These characterizations are then used to design certifying algorithms
to recognize the star $k$-colorable graphs members of these classes of graphs.
Notice that this line of research has been widely considered for other
hereditary properties of graphs, most notoriously graphs with chromatic number
at most $k$ \cite{cameronTCS864,maffraySIDMA26}.   This is the first time that
we are aware of the star chromatic number being studied from the minimal
obstructions perspective. In Section~\ref{sec:chordal} we exhibit the family of
minimal obstructions to star $3$-colorable graphs when restricted to chordal
graphs, and use it to produce a certifying algorithm to recognize this class of
graphs with running time $O(|V|+|E|)$.   Section~\ref{sec:split} is dedicated to
characterize split star $4$-colorable and $5$-colorable graphs in terms of
minimal obstructions; in both cases certifying algorithms of time $O(|V|+|E|)$
are found to recognize the corresponding family of graphs.   The main results of
Section~\ref{sec:2-trees} are a characterization of $2$-paths admitting a star
$4$-coloring and a tight upper bound of $5$ colors on the star chromatic number
of all $2$-paths.

\section{Chordal Graphs}
\label{sec:chordal}

A graph is \textit{chordal} if it has no induced cycle of length at least four.
A vertex $v$ in a graph $G$ is \textit{simplicial} if its neighborhood in $G$ is
a clique.   One of the best known characterizations of chordal graphs affirms
that a graph is chordal if and only if it admits a linear ordering of its vertex
set $v_1 < \cdots < v_n$ such that $v_i$ is a simplicial vertex in $G[\{ v_i,
\dots, v_n\}]$ for each $i \in \{ 1, \dots, n \}$.   Such an ordering is called
a \textit{perfect elimination ordering} of $G$ \cite{golumbic2004}.   Perfect
elimination orderings are important in algorithmic problems related to chordal
graphs; knowing a perfect elimination ordering for $G$ allows to find a maximum
clique of $G$ in time $O(|V|+|E|)$.   When describing algorithms related to star
colorings in chordal graphs we implicitly use the fact that a perfect
elimination ordering of a chordal graph can be found in time $O(|V|+|E|)$.

Let $G$ be a chordal graph and let $(V_1, \dots, V_k)$ be a $k$-coloring of $G$.
Since $G[V_i \cup V_j]$ is both chordal and bipartite, it does not have any
induced cycles.   Hence, in the class of chordal graphs the acyclic chromatic
number coincides with the usual chromatic number. Therefore, within this class,
the star chromatic number is the most interesting of the two parameters.  We
know that a graph admits a star $2$-coloring if and only if it is a forest of
stars.   Thus, we begin our study with the class of chordal graphs admitting a
star $3$-coloring.

\begin{figure}[htb!]
    \centering
    \includegraphics[width=0.82\textwidth]{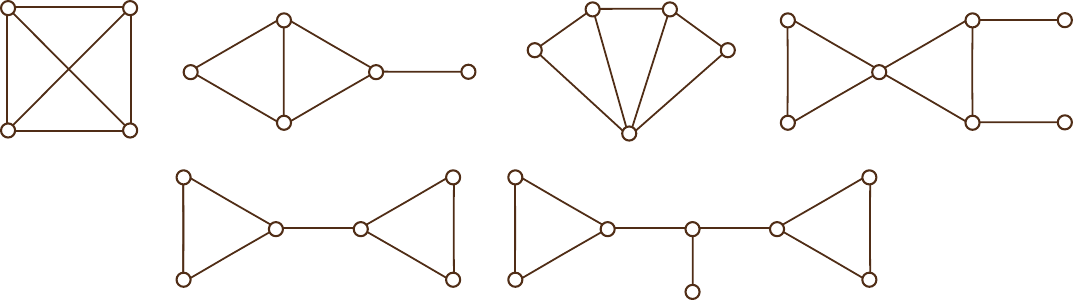}
    \caption{Six minimal obstructions to star $3$-colorability in chordal
    graphs. In the first row, from left to right, we have the graphs $K_4$,
    kite, gem, and $X_{127}$; in the second row the graphs $D_4$ and $D_5$ are
    depicted.}
    \label{fig:chordal_3}
\end{figure}

Let $k$ be a positive integer.   The \textit{$k$-centipede} (also known as
\textit{$k$-comb}) is the graph with vertex set $\{ u_i \}_{i=1}^k \cup \{ v_i
\}_{i=1}^k$ and with edge set $\{ u_i v_i \}_{i=1}^k \cup \{ u_i u_{i+1} \colon\
i \in \{ 1, \dots, k-1 \} \}$.   In other words, the $k$-centipede is obtained
by adding a pendant vertex attached to each vertex in a path of length $k$.
The path $(u_1, \dots, u_k)$ is often called the \textit{spine} of the
centipede.  For an integer $k\ge 4$, we define $D_k$ to be the graph obtained
from the $k$-centipede by identifying $v_1$ with $v_2$ and $v_{k-1}$ with $v_k$
(the lower row of Figure~\ref{fig:chordal_3} shows $D_4$ and $D_5$).  Notice
that $D_4$ is isomorphic to $2K_3 + e$.

\begin{lemma}
\label{lem:dk}
    For every integer $k\ge 4$, the graph $D_k$ is a minimal obstruction to
    star $3$-coloring. 
\end{lemma}
\begin{proof}
    Let $x$ and $y$ be the vertices obtained by identifying $v_1$ with $v_2$ and
    $v_{k-1}$ with $v_k$, respectively.  Let us try to produce a star
    $3$-coloring for $D_k$.   Suppose, without loss of generality, that $c(x) =
    \red$, $c(u_1) = \green$ and $c(u_2) = \blue$. Suppose also, without loss of
    generality, that $c(u_3) = \red$.   Since the sequence of colors of the path
    $(x, u_2, u_3)$ is $(\red, \blue, \red)$, we have that $c(v_3) = c(u_4) =
    \green$.   Now, the sequence of colors of the path $(v_3, u_3, u_4)$ is
    $(\green, \red, \green)$, so it follows that $c(v_4) = c(u_5) = \blue$.
    Inductively, it is clear that $c(v_i) = c(u_{i+1})$ for every $i \in \{3,
    \dots, k-2\}$ and also that the coloring of $(u_3, \dots, u_{k-2})$ is
    forced to be the repeating pattern $(\red, \green, \blue)$.   When $k \ge
    5$, we have that $(v_{k-2}, u_{k-2}, u_{k-1})$ is a path such that
    $c(v_{k-2}) = c(u_{k-1})$ and $\{u_{k-1}, u_k, y\}$ induces a triangle in
    $D_k$. Therefore, if $3$ colors are used to color $D_k$, one of $u_k$ or $y$
    must receive the same color as $u_{k-2}$, which results in a bicolored
    $P_4$.   Hence, $D_k$ is not star $3$-colorable. The case when $k = 4$ is
    similar. 

    To see that $D_4$ is minimal, the reader can refer to the colorings
    presented in Figure \ref{fig:chordal_3-coloring}.   For larger values of
    $k$, the only interesting case is $D_k - v_i$ for $i \in \{3, \dots, k-2\}$.
    The strategy described in the previous paragraph produces the only possible
    star $3$-coloring (modulo a permutation of colors) of the subgraph of $D_k$
    induced by $\{x, u_1, \dots, u_i, v_3, \dots, v_{i-1}\}$.  Thus, after
    coloring with $3$ colors the triangle induced by $\{x, u_1, u_2\}$ and
    propagating the coloring in the only possible way up to $u_i$, we get that
    $c(u_i) = c(v_{i-1})$.   Hence, the color of $u_{i+1}$ is also forced by the
    bicolored path $(v_{i-1}, u_{i-1}, u_i)$, so we get that $c(u_i) \ne
    c(u_{i+1}) \ne c(u_{i-1})$.   Now, we can safely continue with the same
    repeating pattern $(\red, \green, \blue)$ for the path $(u_3, \dots,
    u_{k-2})$, use the color of $u_i$ for $v_{i+1}$ and in general, use the
    color of $u_j$ for $v_{j+1}$ for every $j \in \{i, \dots, k-3\}$.   Notice
    that this coloring can be easily extended to a star $3$-coloring of $D_k -
    v_i$, and it is actually the same coloring that would be obtained by
    coloring first the triangle induced by $\{u_{k-1}, u_k, y\}$ and then
    propagating the coloring with the strategy described in the previous
    paragraph.
\end{proof}

\begin{figure}[htb!]
    \centering
    \includegraphics[width=0.9\textwidth]{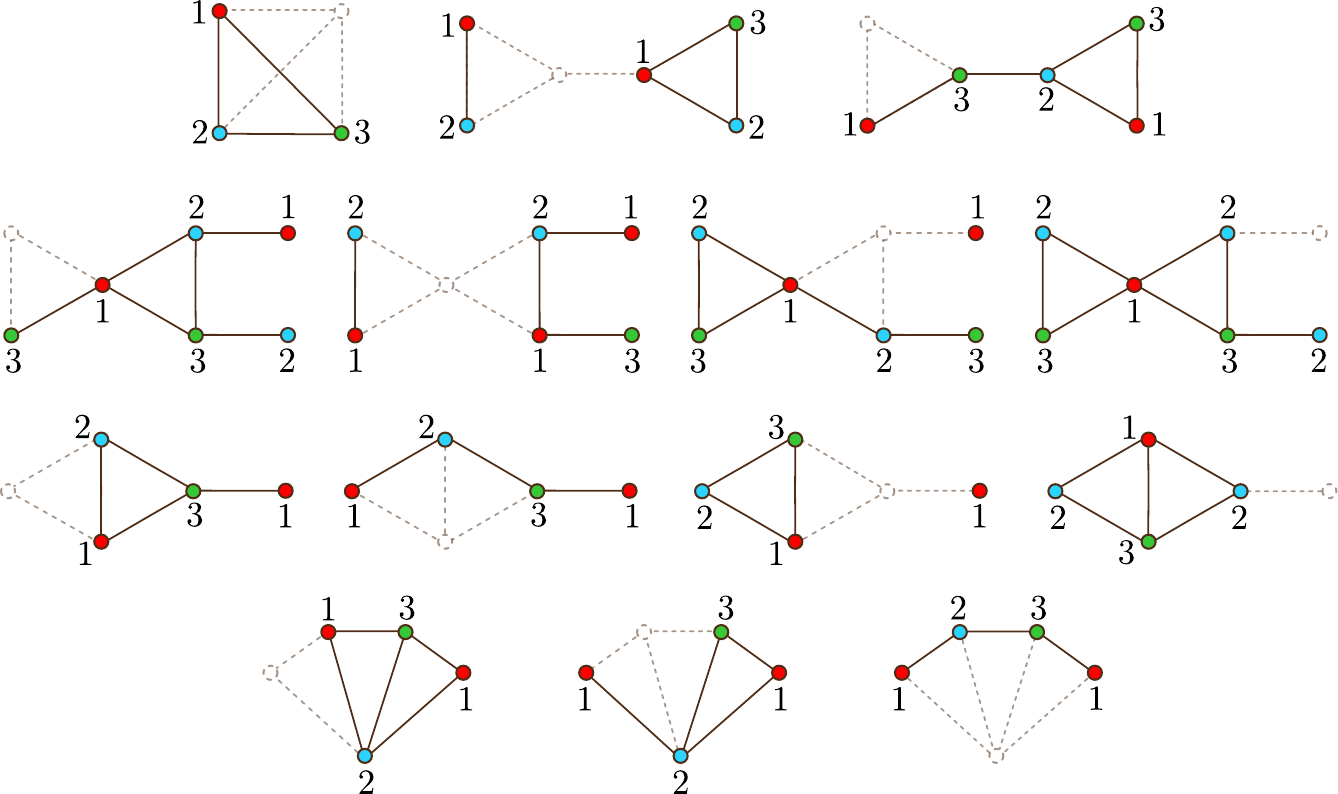}
    \caption{Star $3$-colorings of all induced subgraphs obtained, up to
    isomorphism, by deleting a single vertex from $K_4$, kite, gem, $X_{127}$
    and $D_4$. These colorings show that every proper induced subgraph of
    these graphs is star $3$-colorable. Together with Lemma~\ref{lem:dk}
    (which proves the minimality of the graphs $D_i$ for $i\ge 4$), this
    implies that every graph in $\mathcal{F}_3^{\mathrm{ch}}$ is a minimal
    obstruction to star $3$-colorability in chordal graphs.}
    \label{fig:chordal_3-coloring}
\end{figure}  

It is interesting to observe that the coloring described for $D_k$ when a vertex
is removed in the proof of Lemma \ref{lem:dk} is equivalent, with a suitable
initial choice of colors, to the one obtained by coloring each triangle with
three colors, and then propagating the coloring using BFS until reaching the
only vertex of degree $2$ in the path joining both triangles. 

Set $\mathcal{F}_3^{\mathrm{ch}}$ to be the family $\{K_4, \textnormal{kite},
\textnormal{gem}, X_{127}\} \cup \{D_i \colon\ i \in \mathbb{Z}_{\ge 4}\}$,
partially depicted in Figure~\ref{fig:chordal_3}. Each graph in
$\mathcal{F}_3^{\mathrm{ch}}$ is chordal but does not admit a star $3$-coloring.
Moreover, every proper induced subgraph of a graph in
$\mathcal{F}_3^{\mathrm{ch}}$ is star $3$-colorable, as shown in
Figure~\ref{fig:chordal_3-coloring} and Lemma~\ref{lem:dk}. Thus, the members of
$\mathcal{F}_3^{\mathrm{ch}}$ are minimal obstructions to star $3$-colorability
within the class of chordal graphs. To show that $\mathcal{F}_3^{\mathrm{ch}}$
is, in fact, the complete family of minimal obstructions to star $3$-colorings
in chordal graphs, we first describe the structure of
$\mathcal{F}_3^{\mathrm{ch}}$-free graphs.

For a non-negative integer $n$, the \textit{$n$-crown graph}, or simply
\textit{$n$-crown}, is defined to be the graph $nK_1 \oplus K_2$.   Thus, the
$0$-crown is just $K_2$, the $1$-crown is $K_3$ and the $2$-crown is the
diamond.   We say that a graph is a \textit{crown} if it is an $n$-crown for
some non-negative integer $n$.   The \textit{base} of the crown is the $K_2$ in
the definition of $n$-crown; in a $K_3$, any edge can be the base, but we fix
only one of the edges as the base.   The $0$-crown is referred to as the
\textit{trivial crown}.

For a graph $G$ and an edge $e$ of $G$, an \textit{$e$-augmentation} of $G$ is a
graph obtained from $G$ by adding an independent set of vertices $S$ and adding
the edges from each vertex in $S$ to the endpoints of $e$ (and no other edges).
If $E'$ is a subset of the edge set of $G$, an \textit{$E'$-augmentation} of $G$
is a graph obtained from $G$ by performing successive $e$-augmentations for each
edge $e$ in $E'$.   Notice that an $e$-augmentation can be alternatively
described as pasting a crown with base $e$ along the edge $e$ in $G$.

For a graph $G$ and a vertex $v$ of $G$, a $v$-\textit{triangle substitution} of
$G$ is a graph $G'$ obtained from $G$ by deleting $v$ and adding three new
vertices $v_1, v_2, v_3$, edges $v_1 v_2, v_2 v_3, v_3 v_1$, and for each
neighbor $u$ of $v$ in $G$, exactly one of the edges $uv_1, uv_2$, or $uv_3$.
Notice that it is possible to obtain $G$ from $G'$ by identifying the set
$\{v_1, v_2, v_3\}$.   If $I$ is an independent set of $G$, an
\textit{$I$-triangle substitution} is a graph obtained from $G$ by performing
successive $v$-triangle substitutions for each vertex $v$ in $I$.

Let $F$ be a forest and let $S$ and $I$ be (possibly empty) subsets of $V(F)$
such that $F[S \cup I]$ is a star forest, vertices in $I$ are isolated in $F[S
\cup I]$ and $F[S]$ does not have isolated vertices.   If $F'$ is obtained from
$F$ by performing an $I$-triangle substitution followed by a
$E(F[S])$-augmentation, we say that $F'$ is a \textit{crowned forest}.   Let
$\mathcal{T}$ be the set of all crowned forests.

By construction, every block of a crowned forest $G$ is a crown (possibly a
trivial one), so crowned forests are chordal graphs.   It is not hard to notice
that crowned forests are closed under vertex deletions, so $\mathcal{T}$ is a
hereditary family.   Moreover, given a crowned forest $G$, we can identify a
forest $F$ and the subsets $S$ and $I$ of $V(F)$ used to construct $G$ by the
following procedure.   First, delete all vertices of $G$ which do not belong to
a triangle; this yields a graph such that every connected component either is a
triangle, or contains a universal vertex such that its deletion results in a
star forest. Second, mark the vertices in components that are triangles, and
color red all the vertices in the aforementioned star forest different from the
centre (choose the centre arbitrarily for each $K_2$).   Finally, delete every
red vertex, and successively identify the vertices of each triangle.   The
resulting graph is the forest $F$, the set $I$ is the set obtained from the
triangle identifications, and the set $S$ is the set of neighbors of red
vertices.  As a final observation, notice that there may be more than one way of
constructing a crowned forest from a forest $F$ and sets $S$ and $I$.   Indeed,
when a component in $F[S]$ is isomorphic to $K_2$ and the corresponding edge
augmentation on this copy of $K_2$ creates a block that is a triangle, this
block could also be obtained by a triangle substitution.   It is possible to
restrict the construction of crowned forests so the choice of $F$ and the sets
$S$ and $I$ is unique.

\begin{lemma}
\label{lem:chord-without-D}
    Let $G$ be a connected chordal graph.  The following statements are
    equivalent.
    \begin{enumerate}
        \item $G$ is $\{K_4, \textnormal{kite}, \textnormal{gem}, X_{127},
            D_4\}$-free.
        \item If $G$ is not a trivial graph, then each of its blocks is a crown,
            any two blocks that are non-trivial crowns either share a vertex or
            are at distance at least $2$, and with the only exception of a
            triangle sharing vertices with trivial crowns, any two blocks can
            only share a vertex in their bases.
        \item $G$ is in $\mathcal{T}$.
    \end{enumerate}
\end{lemma}
\begin{proof}
    To prove that the first item implies the second one, let $G$ be a $\{K_4,
    \textnormal{kite}, \textnormal{gem}, X_{127}, D_4\}$-free graph.  Notice
    first that a $K_4$-free chordal graph is a partial $2$-tree, so every block
    of $G$ is a $2$-tree.   Thus, any block of $G$ is either a $K_2$ or is
    obtained starting with a $K_3$ by adding simplicial vertices of degree $2$.
    There are two possibilities for such a construction: The first is to add
    every simplicial vertex with the same set of neighbors, which yields a
    crown; the second is to choose a different edge for some simplicial vertex,
    but such a choice creates a gem, so this case is impossible.   We conclude
    that every block of $G$ is a crown.  It follows from the fact that $G$ is
    $D_4$-free that non-trivial blocks cannot be at distance exactly one, this
    is, they are at distance $0$ (they share a vertex) or they are at distance
    at least $2$.   Since $G$ is $\{ \textnormal{kite}, X_{127} \}$-free, the
    last condition of the second item also holds for the blocks of $G$.

    Suppose now that $G$ is a graph as those described in the second item. Let
    $X_1$ be the set of vertices of $G$ that do not belong to a non-trivial
    crown, i.e., vertices belonging only to blocks isomorphic to $K_2$.    
    Also, let $X_2$ be the set of vertices that are adjacent to both vertices of
    any crown in $G$ and have degree $2$.   As non-trivial blocks are at
    distance different from $1$ and they can only share vertices on their bases,
    it is not difficult to observe that non-trivial blocks in $G - X_2$ do not
    share vertices and each of them is a triangle such that each of its vertices
    is a cut vertex.   Let $T_1, \dots, T_r$ be the non-trivial blocks of $G -
    X_2$. Thus, the graph obtained from $G - X_2$ by recursively identifying the
    vertices in $T_i$, for $i \in \{ 1, \dots, r \}$, is a forest.   So, if we
    let $G'$ be the graph obtained from $G$ by recursively identifying the
    vertices in $T_i$, for $i \in \{ 1, \dots, r \}$, then $G' - X_2$ is a
    forest and $G' - (X_1 \cup X_2)$ is a star forest.   For each $i \in \{ 1,
    \dots, r \}$, let $t_i$ be the vertex of $G'$ obtained from identifying the
    vertices of $T_i$.   From our previous observations we obtain that $\{ t_1,
    \dots, t_r \}$ is an independent set.   So, $G$ is obtained from the forest
    $G' - X_2$ by performing a $\{ t_1, \dots, t_r \}$-triangle substitution
    followed by an $E[G' - (X_1 \cup X_2)]$-augmentation.   Therefore, $G$
    belongs to $\mathcal{T}$.
    
    Finally, it is easy to verify that no graph in $\{K_4, \textnormal{kite},
    \textnormal{gem}, X_{127}, D_4\}$ is in $\mathcal{T}$. Since $\mathcal{T}$
    is a hereditary family, we obtain that the last item implies the first one.
\end{proof}

Before proving the main result of this section, let us briefly discuss simple
strategies for star $3$-coloring some graphs in $\mathcal{T}$. Let $G$ be a
graph in $\mathcal{T}$.   If $G$ is a tree, it is obvious that using BFS,
starting from any vertex, we can obtain a star $3$-coloring for any tree by
using as colors the distances of vertices to the root modulo $3$.   Now, suppose
that $G$ has a single nontrivial block $B$.   If $B$ is a triangle, color it
with colors $\{0, 1, 2\}$ and afterwards, delete the edges of $B$.   We can now
use BFS from each of the vertices of $B$ to color the three remaining trees as
described in the previous case (a shift of the colors is necessary because we
are no longer starting in $0$).   Notice that the vertex colored $i$ in $B$ has
only neighbors of color $i+1$ out of $B$, and its neighbor of color $i+1$ in $B$
has only neighbors of color $i+2$ out of $B$ (notation considered modulo $3$).
Hence, the proposed coloring is a star $3$-coloring.   If $B$ is not a triangle,
since $G$ is kite-free, a similar argument applies for the branchings stemming
from the base of $B$.   There is one case with two nontrivial blocks which is
very similar to the cases with a single nontrivial block.   When there are
exactly two nontrivial blocks $B_1$ and $B_2$ that share a vertex $v$, then we
can again use color $0$ for $v$, color $1$ for the other base vertex in each of
$B_1$ and $B_2$, and color $2$ for the rest of the vertices in these blocks.
From here, we can once again propagate the modulo $3$ distance coloring, and the
fact that $G$ is $\{ \textnormal{kite}, X_{127} \}$-free implies that the only
rooted trees we need to worry about are the ones having as roots $v$ and the two
vertices of color $1$.   Again, this coloring is a star $3$-coloring of $G$. We
use the described strategies for coloring graphs in $\mathcal{T}$ with a single
nontrivial block, or two adjacent nontrivial blocks, in the proof of our
following result.

\begin{theorem}
\label{thm:chord-with-D}
    Let $G$ be a connected chordal graph.  The following statements are
    equivalent.
    \begin{enumerate}
        \item $G$ is $\mathcal{F}_3^{\mathrm{ch}}$-free.
        \item If $G$ is not a trivial graph, then
            \begin{itemize}
                \item each of its blocks is a crown,
                \item two blocks that are nontrivial crowns either share a
                    vertex or are at distance at least $2$,
                \item except in the case of a triangle sharing vertices with
                    trivial blocks, two blocks can only share a vertex in their
                    bases, and
                \item the path joining two nontrivial blocks at positive
                    distance contains at least one vertex of degree $2$.
            \end{itemize}
        \item $\x(G) \le 3$.
    \end{enumerate}
\end{theorem}
\begin{proof}
    The implication from the first item to the second one follows directly from
    Lemma~\ref{lem:chord-without-D} and the structure of $D_k$ for $k\ge 4$.

    To prove that the second item implies the third one, we use induction on the
    number of nontrivial blocks.   If $G$ has at most one nontrivial block or
    exactly two adjacent nontrivial blocks, we use the strategy described in the
    paragraph before this theorem to obtain a star $3$-coloring of $G$.
    Otherwise, by choosing a vertex $v$ of degree $2$ minimizing the number of
    nontrivial blocks in one of the components of $G-v$, we find that such a
    component has at most one nontrivial block or exactly two nontrivial blocks
    at distance $0$. (Otherwise, each component of $G-v$ has at least two
    nontrivial components at positive distance and the path between them has a
    vertex $w$ of degree $2$; one of the two components of $G-w$ has less
    nontrivial blocks than $G-v$.)   Let $C_1$ and $C_2$ be the components of
    $G-v$, where $C_2$ is the component having at most one non-trivial block $B$
    or two adjacent nontrivial blocks $B_1$ and $B_2$. By induction hypothesis,
    $G[V(C_1) \cup \{v\}]$ has a star $3$-coloring $c_1$.   Suppose without loss
    of generality that $c_1$ assigns color $0$ to $v$ and color $1$ to its only
    neighbor in $C_1$.   By coloring $C_2$ with the strategy described in the
    paragraph previous to this result, and maybe after a permutation of colors,
    we obtain a coloring $c_2$ of $G[V(C_2) \cup \{v\}]$ that assigns color $0$
    to $v$ and color $2$ to its only neighbor in $C_2$.   Therefore, $c_1 \cup
    c_2$ is a star $3$-coloring of $G$.

    Finally, we have already observed that no graph in
    $\mathcal{F}_3^{\mathrm{ch}}$ is star $3$-colorable, so it follows that the
    last item implies the first.
\end{proof}

We finalize this section by using the structural information obtained so far to
present an efficient recognition algorithm for star $3$-colorable chordal
graphs.

\begin{theorem}
\label{thm:chord-algo}
    There is a certifying algorithm with running time $O(|V|+|E|)$ to determine
    whether an input chordal graph $G$ is star $3$-colorable.
\end{theorem}
\begin{proof}
    We first use $O(|V|+|E|)$ time to find the maximum clique of $G$. If it has
    more than $3$ vertices, we return $K_4$ as a no-certificate. Afterwards, we
    find the blocks of $G$ in $O(|V|+|E|)$ time (using, e.g., Hopcroft and
    Tarjan's algorithm \cite{hopcroftCACM16}), and inspect each of them to
    verify that it is a crown. The latter can also be achieved in time
    $O(|V|+|E|)$ by checking that each block has two or three vertices (it is a
    $0$-crown or a $1$-crown), or that within each block every vertex has degree
    $2$, except maybe for two of them, which are in the base of the crown; if
    this check fails, we find a triangle with each vertex of degree greater than
    $2$, and use it to return either a gem or a kite as a no-certificate. While
    we carry out this process, we create a list $B_1, \dots, B_r$ of all the
    nontrivial blocks of $G$ keeping constant-time access to each of them.

    The aforementioned routine leaves us with a graph in which each nontrivial
    block is a crown.    The algorithm has an additional preprocessing step
    before the actual coloring happens.   For the purposes of this proof, when
    we say ``nontrivial block'' we also consider pairs of adjacent nontrivial
    blocks.   Notice that by starting a BFS search from each nontrivial block
    and stopping a branch when we reach a vertex of degree $2$, we either
    explore a subgraph of $G$ having a single nontrivial block, or we obtain a
    branch between two nontrivial blocks where each intermediate vertex has
    degree at least $3$.   In the latter case we choose a triangle in each of
    these blocks together with the path $P$ joining them, and add an additional
    neighbor of each internal vertex of $P$; the resulting subgraph is
    isomorphic to $D_i$ for some integer $i\ge 4$ and is returned as a
    no-certificate. In the former case, assuming that we started from block
    $B_i$, we add $i$ to the label of the vertices of degree $2$ where the
    search stopped.   When this process finishes, vertices of degree $2$ reached
    by a single nontrivial block $B_i$ will have label $\{i\}$, and vertices
    reached by two different blocks $B_i$ and $B_j$ will have label $\{i,j\}$,
    for some $i,j \in \{1, \dots, r\}$.   Keep in mind that the search from each
    block stops at a branch when the first vertex of degree $2$ is found, and
    saves the distance from the block to each of these vertices (the distance
    could also be saved in the label).

    Now we start to color $G$ using the strategies described in the paragraph
    before Theorem~\ref{thm:chord-with-D}.  Color $B_1$ with $3$ colors and
    propagate this coloring following BFS until all vertices of degree $2$
    having $1$ in their label are colored; keep these vertices in a queue $Q$.
    If the next vertex $v$ in $Q$ has label $\{1,j\}$, then process $B_j$ in the
    same way as $B_1$, i.e., color it with $3$ colors and propagate the coloring
    until all vertices of degree $2$ having $j$ in their label are colored;
    enqueue these vertices as the coloring finds them.  Notice that we know the
    distance from $B_j$ to $v$, so we can choose the coloring of $B_j$ and the
    order of the colors used in the propagation in such a way that this coloring
    of $B_j$ assigns $v$ the color it already has, and more important, it
    assigns the neighbor of $v$ on the $B_j$ side a different color from the
    neighbor of $v$ on the $B_1$ side.   Therefore, we successfully extended the
    star $3$-coloring to now include an additional nontrivial block.   If the
    label of $v$ is $\{1\}$, then we keep on propagating this coloring until we
    reach vertices of degree $2$ with label $\{j\}$, add $1$ to their label, and
    enqueue them.  Notice that in this step the coloring cannot reach a vertex
    of degree $2$ with label $\{i,j\}$, as otherwise vertex $v$ should have
    label $\{1,i\}$.   This process is repeated until $Q$ is empty, i.e., until
    all labeled vertices of degree $2$, and hence all nontrivial blocks, have
    been colored.

    Every time we process a vertex $v$ of degree $2$ with label $\{i,j\}$, we
    are extending the current partial star $3$-coloring to include an additional
    nontrivial block.   In each step, we are moving away an additional step from
    $B_1$ in a BFS fashion (thanks to the use of $Q$), and we only have to worry
    about exactly two colorings to be compatible.

    The algorithm uses time $O(|V|+|E|)$ to find all the blocks and check that
    each of them is a crown.   In the second step, when we use BFS from every
    nontrivial block, each vertex of the graph is traversed at most once, so
    this step also takes time $O(|V|+|E|)$.   In the coloring step, labeled
    vertices of degree $2$ are explored twice, but every other vertex is
    explored once.   Since we are using BFS for the coloring, the time of this
    step is also $O(|V|+|E|)$, which is also the overall running time of the
    complete algorithm.
\end{proof}

\section{Split graphs}
\label{sec:split}

A graph $G$ is a \textit{split graph} if its vertex set can be partitioned into
a clique $K$ and an independent set $S$. Such a partition $(K, S)$ is called a
\textit{split partition} of $G$. We denote by $\omega(G)$ the clique number of
$G$, that is, the maximum size of a clique in $G$. Throughout this paper,
whenever we fix a split partition $(K,S)$ of a split graph $G$, we assume that
$K$ is a maximal clique of $G$, so $|K| = \omega(G)$.

It is well known that split graphs are perfect, since they are both chordal and
co-chordal; in particular, their chromatic number equals their clique number:
\[
  \chi(G)=\omega(G)=|K| \qquad \text{for every split graph $G$.}
\]

Notice that for every $C_4$-free graph $G$ a proper coloring $c$ of $G$ is also
a star coloring of $G$ if and only if no \textbf{induced} path $P_4$ in $G$ is
bicolored.  This is true for split graphs, since they are precisely the $\{2K_2,
C_4, C_5\}$-free graphs \cite{foldes1977}. Thus, to verify that a proper
coloring $c$ of $G$ is a star coloring, it is enough to check that no induced
$P_4$ is bicolored.  In addition, it is easy to check that every induced path
$(x_1, x_2, x_3, x_4)$ in a split graph $G = (K, S)$ satisfies $x_1,x_4 \in S$
and $x_2,x_3 \in K$.

We start by delimiting the only two possible values for the star chromatic
number of a split graph. Although this fact is known, we include a proof for the
sake of completeness. 
\begin{proposition}
    \label{pro: split omega or omega+1}
    If $G = (K, S)$ is a split graph, then
    \(
        \x(G)\in \{\omega(G),\omega(G)+1\}.
    \)
\end{proposition}

\begin{proof}
    To prove that $\x(G) \le \omega(G) +1$, define a proper coloring $c$ of $G$
    by assigning a distinct color to each vertex of $K$ and using one additional
    color for all vertices in $S$.   Let $P = (x_1, x_2, x_3, x_4)$ be an
    induced path in $G$. Since $x_1, x_4 \in S$ and $x_2, x_3 \in K$, we have
    that $c(x_1) = c(x_4)$ and $c(x_2) \neq c(x_3)$, so $P$ is colored with at
    least three distinct colors. Hence, no induced $P_4$ in $G$ is bicolored
    under $c$, so $c$ is a star coloring and $\x(G) \le \omega(G)+1$. From here,
    the result is a consequence of \eqref{eq:inequalitiesChain}.
\end{proof} 

The following simple observation allows us to assume that neighborhoods
in a split minimal obstruction to star $k$-colorability are pairwise
incomparable under inclusion. As usual, for a given vertex $v \in V(G)$, $N(u)$
denotes its neighborhood.

\begin{proposition}
\label{pro:neighborhood-reduction}
    Let $G$ be a split graph with split partition $(S,K)$. If $u,v \in S$ are
    such that $N(u)\subseteq N(v)$, and $c$ is a star $k$-coloring of $G-u$,
    then $c$ can be extended to a star $k$-coloring of $G$ by setting
    $c(u)=c(v)$. Consequently,
    \[
        \x(G)=\x(G-u).
    \]
    In particular, if $G$ is a split minimal obstruction to star
    $k$-colorability, then the family of neighborhoods $\{N(x)\colon x\in S\}$
    is an antichain under set inclusion.
\end{proposition}

As an immediate consequence of Proposition~\ref{pro:neighborhood-reduction}, for
any positive integer $k$ there are finitely many split minimal obstructions to
star $k$-colorability.

The following result provides a simple sufficient condition for a split graph to
satisfy $\x(G)=\omega(G)$.

\begin{proposition}
\label{pro:split-K-insaturated}
    Let $G$ be a split graph with split partition $(K, S)$. If there exists a
    vertex $y\in K$ that is not adjacent to any vertex in $S$, then
    $\x(G)=\omega (G)$. 
\end{proposition}

\begin{proof}
    Assign distinct colors to the vertices of $K$ and color every vertex of $S$
    with the color assigned to $y$. Since $y$ has no neighbors in $S$, this
    coloring is proper. Also, any two color classes induce a star forest, and
    hence the coloring is a star $\omega(G)$-coloring. Thus $\x(G)\le
    \omega(G)$, while the reverse inequality follows from $\omega(G)\le \x(G)$.
\end{proof} 

We next identify some configurations that prevent a split graph from having star
chromatic number equal to its clique number.

\begin{proposition}
\label{pro:omega-obstructions}
    Let $G$ be a split graph with split partition $(K,S)$ and let $\omega =
    \omega(G) = |K|$. If $G$ satisfies any of the following conditions, then 
    $\x(G)=\omega(G)+1$.
    \begin{enumerate}[label=(\roman*)]
        \item There exist vertices $x_1, x_2\in S$ such that the union of their
            neighborhoods contains $K$. 
        \item There is a partition $K=K_1\cup K_2$ and distinct vertices $u, v,
            w, z\in S$ such that
            \begin{itemize}
                \item $u$ is adjacent to all vertices of $K_1$, and
                \item there is an antimatching between $\{v,w,z\}$ and $K_2$,
                    that is, there exist pairwise distinct vertices $y_v, y_w,
                    y_z\in K_2$ such that each of $v, w, z$ is adjacent to every
                    vertex of $K_2$ except the corresponding one in $\{y_v, y_w,
                    y_z\}$;
            \end{itemize}
            no further restrictions are imposed on the adjacencies of $G$.

        \item There is a partition $K=K_1\cup K_2$ and distinct vertices $u,
            x_0, x_1, x_2, x_3, x_4\in S$ such that 
            \begin{itemize}
                \item $u$ is adjacent to all vertices of $K_1$, and 
                \item there exist distinct vertices $y_0, y_1, y_2, y_3, y_4\in
                    K_2$ for which, for every $i \in \{0, 1, 2, 3, 4\}$, $x_i$
                    is adjacent to every vertex of $K_2$ except $y_{i-2}$ and
                    $y_{i+2}$, where subscripts are considered modulo $5$;
            \end{itemize}
            no further restrictions are imposed on the adjacencies of $G$.
            (See Figure~\ref{fig:split_omegaminusthree}.)
    \end{enumerate}
\end{proposition}

\begin{proof}
    Suppose for a contradiction that $G$ admits a star coloring $c$ with exactly
    $\omega$ colors. Since $K$ is a clique of size $\omega$, every vertex in $K$
    must receive a different color. In particular, every color of $c$ appears
    exactly once on $K$. Consequently, for every vertex $x\in S$, the color
    $c(x)$ must coincide with the color of a unique vertex of $K$. We will show
    that under any of the conditions $(i)-(iii)$, we can find a bicolored $P_4$,
    contradicting that $c$ is a star coloring. 

    \medskip
    \noindent\textbf{Case $(i)$.} Set $K=\{y_1,\ldots,y_{\omega}\}$. Since $K
    \subseteq N(x_1)\cup N(x_2)$, by relabelling the vertices of $K$ we may
    assume that there exist integers $i$ and $j$ such that $1\le i\le j\le
    \omega$, $N(x_1) = \{y_1,\dots,y_j\}$ and $N(x_2) = \{y_i,\dots,y_\omega\}$.
    Thus, $c(x_1)=c(y_p)$ for some $j<p\le \omega$ and $c(x_2)=c(y_q)$ for some
    $1\le q<i$, and therefore the path $(x_1, y_q, y_p, x_2)$ is bicolored, a
    contradiction. 

    \medskip
    \noindent\textbf{Case $(ii)$.} Let $y \in K$ be the unique vertex with $c(y)
    = c(u)$. Notice that, since $u$ is adjacent to all vertices in $K_1$, we
    must have $y\in K_2 \setminus N(u)$. Without loss of generality, assume $y
    \notin \{y_v,y_w\}$. If $c(v)=c(y')$ for some $y'\in K_1$, then $(u, y', y,
    v)$ is a bicolored $P_4$ in $G$, which is impossible. Thus, it must be the
    case that $c(v) = c(y_v)$, and similarly, $c(w) = c(y_w)$. However, in this
    case the path $(v, y_w, y_v, w)$ is a bicolored $P_4$, a contradiction. 

    \medskip
    \noindent\textbf{Case $(iii)$.} Let $X$ and $Y$ be the sets $X = \{x_0, x_1,
    x_2, x_3, x_4\}$ and $Y = \{y_0, y_1, y_2, y_3, y_4\}$. Let $y\in K$ be the
    unique yertex with $c(u) = c(y)$ so, as in the previous case, $y \in K_2
    \setminus N(u)$. Without loss of generality assume 
    \[
        c(y_0)=\red, \quad c(y_1)=\blue, \quad c(y_2)=\green, \quad
        c(y_3)=\purple, \quad c(y_4)=\orange.
    \] 

    \begin{itemize}
    \item First, suppose that $y\notin Y$, so $y$ is adjacent to all vertices in
        $X$. Note that for every $x\in X$ and every $z\in K_1$ we have $c(x)\ne
        c(z)$, since otherwise the path $(u, z, y, x)$ would be an induced
        bicolored $P_4$. Consequently, for each $i$ the color $c(x_i)$ must be
        the color of some yertex in $K_2$. By the neighborhood condition on
        $x_i$ in $K_2$, this forces
        \[
            c(x_i)\in\{c(y_{i-2}),c(y_{i+2})\}.
        \]
        Without loss of generality, assume $c(x_0)=c(y_2)=\green$.  We have
        $c(x_2)\in\{\red,\orange\}$. If $c(x_2)=\red$, then the path
        $(x_0,y_0,y_2,x_2)$ is bicolored; if $c(x_2)=\orange$, then
        $(x_0,y_4,y_2,x_2)$ is bicolored. In both cases we obtain a
        contradiction.

    \item Suppose now that $y\in Y$. By symmetry, we may assume $y=y_0$ and
        hence $c(u)=c(y_0)=\red$. For $i\in\{0,1,4\}$ and any $z\in K_1$, we
        have $c(x_i)\ne c(z)$; otherwise $(u, z, y_0, x_i)$ would be an induced
        bicolored $P_4$, which is impossible. Thus,
        \[
            c(x_0)\in \{\green,\purple\},\quad c(x_1)\in \{\purple,\orange\},
            \quad \text{and}\quad c(x_4)\in \{\blue,\green\}.
        \]

        \begin{itemize}
            \item If $c(x_0)=\green$, then $c(x_1)\ne\orange$, since otherwise
                $(x_0,y_4,y_2,x_1)$ is bicolored. Hence $c(x_1)=\purple$. Now if
                $c(x_4)=\blue$, then $(x_1, y_1, y_3, x_4)$ is bicolored, while
                if $c(x_4)=\green$, then $(x_1, y_2, y_3, x_4)$ is bicolored.
                Thus $c(x_4) \notin \{\blue,\green\}$, which is absurd.

            \item If $c(x_0)=\purple$, then $c(x_4)\ne\blue$, as otherwise
                $(x_0, y_1, y_3, x_4)$ is bicolored. Hence $c(x_4)=\green$. Now
                if $c(x_1)=\purple$, then $(x_1, y_2, y_3, x_4)$ is bicolored,
                while if $c(x_1)=\orange$, then $(x_1, y_2, y_4, x_4)$ is
                bicolored. Thus $c(x_1) \notin \{\purple,\orange\}$, a
                contradiction.
        \end{itemize}
    \end{itemize}
\end{proof}

\begin{figure}
    \centering
    \includegraphics[width=0.8\textwidth]{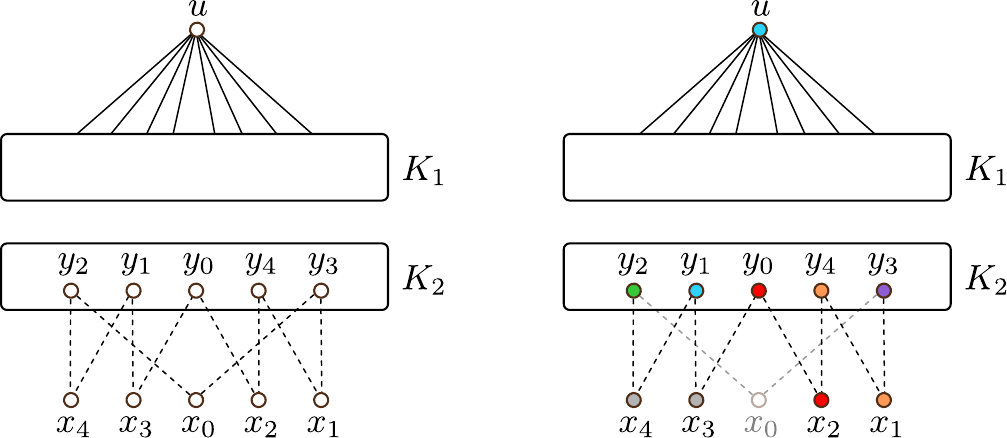}
    \caption{Left: the configuration described in item~(iii) of
    Proposition~\ref{pro:omega-obstructions}. Right: minimality of the
    obstruction where, without loss of generality, we delete the vertex $x_0$.
    By item~(i) of Proposition~\ref{pro:omega-obstructions}, we may assume that
    there exists a vertex $z\in K_1$ that is adjacent to neither $x_3$ nor $x_1$
    (otherwise $N(x_3)\cup N(x_1)$ would cover $K$), and a vertex $z'\in K_1$
    (possibly $z=z'$) adjacent to neither $x_4$ nor $x_1$. Accordingly, in the
    depicted coloring we set $c(x_3)=c(z)$ and $c(x_4)=c(z')$.} 
    \label{fig:split_omegaminusthree}
\end{figure}

By Proposition~\ref{pro:neighborhood-reduction},
Proposition~\ref{pro:split-K-insaturated}, and item~(iii) of
Proposition~\ref{pro:omega-obstructions}, we obtain the following result.

\begin{corollary}
    Let $G$ be a split graph with split partition $(K,S)$.   If there are at
    least three vertices of degree $\omega - 2$ in $S$, then it is possible to
    compute $\x$ in time $O(|V|+|E|)$.
\end{corollary}
\begin{proof}
    Let $v_1, \dots, v_\omega$ be the vertices in $K$ and assume that vertex
    $v_i$ receives color $i$ in a partial star coloring of $G$.   If there is a
    vertex $u$ in $S$ of degree $\omega - 1$, then $\chi_\star(G) = \omega$ if
    and only if no vertex in $S$ is adjacent to the only vertex in $K \setminus
    N(u)$. These two conditions can be verified in time $O(|V|+|E|)$ by
    exploring the neighborhoods of vertices in $S$, and an optimal star
    colouring of $G$ is obtained by using either the available color or a new
    color for every vertex in $S$.

    Suppose now that $\Delta = \omega - 2$.   We focus on three vertices, $u,v$
    and $w$ of degree $\omega - 2$ and test in time $O(|V|+|E|)$ whether there
    are two of them such that the union of their neighbourhoods contains $K$. If
    not, suppose first, and without loss of generality, that $v_\omega$ is not
    in the union of the neighborhoods of $u, v$ and $w$.   This means that each
    of $u, v$ and $w$ are adjacent to all vertices in $\{ v_1, \dots, v_{\omega
    - 1} \}$ but one; suppose without loss of generality that $u$ is not
    adjacent to $v_1$, $v$ is not adjacent to $v_2$, and $w$ is not adjacent to
    $v_3$. If some vertex $x$ in $S$ is adjacent to $v_\omega$, then by setting
    $K_1 = \{ v_\omega \}$ and $K_2 = \{ v_1, \dots, v_{\omega - 1} \}$ we have
    that $x$ is complete to $K_1$ and there is an antimatching between $\{u, v,
    w\}$ and $K_2$, so $\chi_\star(G) = \omega + 1$, and thus a new color can be
    used for every vertex in $S$.   If $v_\omega$ is not a neighbor of any
    vertex in $S$, then as in Proposition~\ref{pro:split-K-insaturated} we use
    the color $\omega$ for every vertex in $S$ to obtain a star
    $\omega$-coloring of $G$.   Else, $K$ is contained in the union of $N(u),
    N(v)$ and $N(w)$ but not in the union of two of them.   But this can only
    happen if $|N(u) \cap N(v) \cap N(w)| = \omega - 3$. Suppose without loss of
    generality that $K \setminus (N(u) \cap N(v) \cap N(w)) = \{ v_1, v_2, v_3
    \}$, and each of $u, v$ and $w$ is adjacent to exactly one of $\{ v_1, v_2,
    v_3 \}$ (again, without loss of generality, $uv_1, vv_2$ and $wv_3$ are
    edges of $G$).   Now, the algorithm classifies each of the rest of the
    vertices in $S$ depending on which of $v_1, v_2$ or $v_3$ it is adjacent to.
    If there is a vertex $x$ which is adjacent to at least $2$ of these
    vertices, then we are again in the case of the first item of
    Proposition~\ref{pro:omega-obstructions}, and a new color can be used to
    color every vertex in $S$.   Else, each of the vertices in $S$ is adjacent
    to at most one of $v_1, v_2$ or $v_3$, and thus, its neighborhood is
    contained in $N(u), N(v)$ or $N(w)$.   It follows from
    Proposition~\ref{pro:neighborhood-reduction} that in this case is enough to
    colour $G[K \cup \{ u, v, w \}]$ and then assign one of the colors used for
    $u, v$ or $w$ for each of the remaining vertices in $S$ according to which
    of $v_1, v_2$ or $v_3$ it is adjacent to (the algorithm has already done
    this classification).   A valid star $\omega$-coloring for $G[K \cup \{ u,
    v, w \}]$ assigns color $2$ to $u$, color $3$ to $v$ and color $1$ to $w$.
    Notice that the classification of vertices in $S$ can be done in time
    $O(|V|+|E|)$ by exploring the neighborhoods first of $u, v$ and $w$ and then
    of the rest of the vertices of $S$.   Similarly, the described coloring can
    be constructed in the same running time.

    Finally, if there are two vertices, say $u$ and $v$, such that $K \subseteq
    N(u) \cup N(v)$, then the first item of
    Proposition~\ref{pro:omega-obstructions} implies that $\chi_\star(G) =
    \omega + 1$, so a new color is used for every vertex in $S$.
\end{proof}

Two vertices $u$ and $v$ of a graph $G$ are said to be \textit{true twins} if
they are adjacent and $N(u) \setminus \{v\} = N(v) \setminus \{u\}$. The
following proposition shows that the creation of some true twins in a split
minimal obstruction to star $k$-coloring yields minimal obstructions to star
$(k+1)$-coloring.

\begin{proposition}
\label{pro:add true twins}
    Let $G$ be a split graph, with split partition $(K, S)$, and let $v$ be a
    vertex in $K$. Suppose that $G'$ is obtained from $G$ by adding a new vertex
    $v'$ such that $v$ and $v'$ are true twins in $G'$.  If $G$ is a minimal
    obstruction to star $k$-coloring, with $k = \omega(G)$, then $G'$ is a
    minimal obstruction to star $(k+1)$-coloring.
\end{proposition}

\begin{proof}
    Note that $G'$ is also a split graph with split partition $(K\cup \{v'\},
    S)$. Proceeding by contradiction, suppose that $G'$ admits a star
    $(k+1)$-coloring $c'$, and for each $i\in \{1, 2, \ldots, k+1\}$, let $V'_i$
    be the color class of color $i$ in $G'$; since $\omega (G')=k+1$, we have
    that $|V'_i\cap (K\cup \{v'\})|=1$.  Without loss of generality, assume that
    $c'(v)=k$ and $c'(v')=k+1$. Let $c \colon V(G) \to \{1, 2, \ldots, k\}$ be
    the coloring of $G$ defined as $c(x) = \min\{c'(x), k\}$ for each $x\in
    V(G)$.    
    
    Let us first show that $c$ is a proper coloring. For every $j\in \{1, 2,
    \ldots, k\}$, let $V_i$ be the color class of color $i$ in $G$ with respecto
    to $c$. Note that, for every $i \le k-1$, $V_i = V'_i$, so $V_i$ is an
    independent set. In addition, $V_k\cap K = \{v\}$, $(V_k\setminus \{v\})
    \subseteq S$, and because $v$ and $v'$ are true twins in $G'$, we have that
    $V_k$ is an independent set of $G$, and $c$ is a proper $k$-coloring of $G$.
 
    We now prove that $c$ is a star coloring. Let $P = (x_1, x_2, x_3, x_4)$ be
    an induced $P_4$ in $G$.  As we observed before, $x_1,x_4\in S$, and
    $x_2,x_3\in K$, and $P$ clearly is an induced $P_4$ also in $G'$.  Suppose
    that $P$ is bicolored in $G$; it follows that $x_1$ and $x_3$ have the same
    color, and $x_2$ and $x_4$ have the same color.  Since $V_k\cap K = \{v\}$,
    if $v \notin V(P)$, then no vertex in $P$ has color $k$, and $P$ is also
    bicolored in $G'$, which is impossible.  Hence, $v$ is a vertex of $P$, let
    us say, without loss of generality $v=x_2$. It follows that $x_4$ has color
    $k$, and $x_1$ and $x_3$ have color $i\neq k$ in $G$.  Observe that $x_4$
    has color $k+1$ in $G'$, otherwise $P$ is a bicolored $P_4$ in $G'$, which
    is absurd.  But, since $v$ and $v'$ are true twins in $G'$, we have that
    $(x_1, v', x_3, x_4)$ is a bicolored $P_4$, which is impossible.  Thus,
    there is no bicolored $P_4$ in $G$, and we conclude that $c$ is a star
    $k$-coloring of $G$, which is a contradiction.  Therefore, it must be the
    case that $G'$ is an obstruction to star $(k+1)$-coloring.

    To prove that $G'$ is minimal, consider an arbitrary vertex $x\in V(G')$. If
    $x=v'$, then $G'-v'=G$ which is star $(k+1)$-colorable. Now, assume that
    $x\neq v'$. Since $G$ is a minimal obstruction to star $k$-coloring, there
    is a star $k$-coloring $c^*$ of $G-x$, and then, $c^* \cup \{(v', k+1)\}$ is
    a star $(k+1)$-coloring of $G'-x$.  Therefore, $G'$ is a minimal obstruction
    to star $(k+1)$-coloring.
\end{proof} 

\begin{corollary}
    For every $k\geq 3$, there are at least as many minimal obstructions to a
    star $(k+1)$-coloring as there are minimal obstructions to a star
    $k$-coloring.
\end{corollary}

As in Proposition~\ref{pro:add true twins}, the following result provides a
way to construct a minimal obstruction to star $(k+1)$-colorability from a
minimal obstruction to star $k$-colorability.

\begin{proposition}
\label{pro: split add universal vertex}
Let $G$ be a split graph with split partition $(K,S)$. Suppose that $G'$ is
obtained from $G$ by adding a new universal vertex $w$. If $G$ is a minimal
obstruction to star $k$-colorability, with $k=\omega(G)$, then $G'$ is a
minimal obstruction to star $(k+1)$-colorability.
\end{proposition}

Obviously, any graph $G$ is star $2$-colorable if and only if it is a star
forest.   It is also clear that these two conditions occur if andd only if $G$
is a $\{K_3, P_4, C_4\}$-free graph.   Thus, there is a certifying algorithm
based in BFS to decide whether an input graph $G$ admits a star $2$-coloring
that runs in time $O(|V|+|E|)$.   It follows from Theorem~\ref{thm:chord-with-D}
that a split graph $G$ is star $3$-colorable if and only if $G$ is a $\{K_4,
\textnormal{kite}, \textnormal{gem}\}$-free graph. Moreover, by
Theorem~\ref{thm:chord-algo} we have that there is a $O(|V|+|E|)$-time
certifying algorithm to determine whether a split graph is star $3$-colorable.

Let $\mathcal{F}_4^{\mathrm{sp}}$ be the set depicted in
Figure~\ref{fig:split_mo_K=4}. Note that every graph in
$\mathcal{F}_4^{\mathrm{sp}}$ is split but, by
Proposition~\ref{pro:omega-obstructions}, does not admit a star $4$-coloring. On
the one hand, it is clear that $K_5$ is a minimal obstruction to star
$4$-coloring.  On the other hand, if any vertex $v$ of the clique is removed,
then, by Proposition ~\ref{pro: split omega or omega+1}, $G_i-v$ has a star
$4$-coloring, for $i\in \{1, \ldots, 5\}$.  For every $i \in \{1, 2, 3, 4\}$, if
any vertex $v$ of the stable set is removed, then there is at least one vertex
in the clique of $G_i-v$ that is no adjacent to any vertex in the stable set, so
by Proposition~\ref{pro:split-K-insaturated}, $G_i-v$ has a star $4$-coloring.
For $G_5$, suppose that $V(G_5) = \{v_1, v_2, v_3, v_4, u_1, u_2, u_3, u_4\}$,
where $K = \{v_1, \ldots, v_4\}$, $S = V(G) \setminus K$, and the unique vertex
adjacent to $u_1$ is $v_1$.  Note that, by
Proposition~\ref{pro:split-K-insaturated}, $G_5-u_1$ has a star $4$-coloring.
Now, assume that $u_2$ is adjacent to $v_2$ and $v_3$, $u_3$ is adjacent to
$v_2$ and $v_4$, and $u_4$ is adjacent to $v_3$ and $v_4$.  
We define the vertex coloring $c$ of $G_5-u_2$ as follows: $c(u_1) = 3$, $c(u_3)
= 1$, $c(u_4) = 2$, and for each $i \in \{1, 2, 3, 4\}$, $c(v_i) = i$.  It is
easy to prove that $c$ is a star $4$-coloring of $G_5-u_2$. Moreover, since
$G_5-u_2$, $G_5-u_3$ and $G_5-u_4$ are pairwise isomorphic, it follows that
every proper induced subgraph of $G_5$ is $4$-colorable.  Therefore, every graph
in $\mathcal{F}_4^{\mathrm{sp}}$ is a minimal obstruction to star
$4$-colorability within the class of split graphs.  In the next proposition we
prove that the graphs in $\mathcal{F}_4^{\mathrm{sp}}$ are the only split graphs
that are minimal obstructions to star $4$-colorability.

\begin{figure}[htb!]
    \centering
    \includegraphics[width=0.9\textwidth]{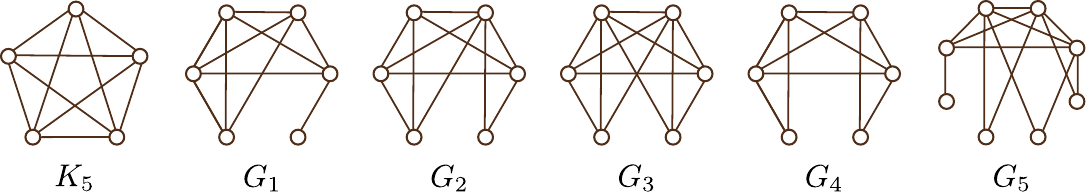}
    \caption{Minimal obstructions to star $4$-coloring in split graphs. } 
    \label{fig:split_mo_K=4}
\end{figure}

\begin{proposition}
\label{pro: star $4$-col iff F_4-free}
    If $G$ is a split graph, then $G$ is star $4$-colorable if and only if $G$
    is $\mathcal{F}_4^{\mathrm{sp}}$-free.
\end{proposition}

\begin{proof}
    Let $G$ be a split graph with split partition $(K,S)$. Since each element of
    $\mathcal{F}_4^{\mathrm{sp}}$ is a minimal obstruction to star
    $4$-coloring, it follows that if $G$ is star $4$-colorable, then $G$ is
    $\mathcal{F}_4^{\mathrm{sp}}$-free.

    Now, suppose that $G$ is $\mathcal{F}_4^{\mathrm{sp}}$-free. To prove that
    $G$ is $4$-star colorable, it is enough to show that every connected
    component of $G$ is, so we can assume that $G$ is connected. Since $K_5$ is
    not an induced subgraph of $G$, it follows that $|K|\leq 4$. By
    Proposition~\ref{pro: split omega or omega+1}, if $|K|\leq 3$, then $G$ is
    star $4$-colorable. Hence, assume that $|K|=4$, and $K=\{0,1,2,3\}$. By
    Proposition~\ref{pro:split-K-insaturated}, if there is a vertex in $K$ that
    is not adjacent to any vertex in $S$, then $G$ is star $4$-colorable. Thus,
    we can assume that $N(S)=K$. Since $K_5$ is not an induced subgraph of $G$
    and $G$ is connected, we have that $1\leq d_G (x)\leq 3$ for every $x\in S$.
    This implies $|S|\geq 2$. Moreover, by
    Proposition~\ref{pro:neighborhood-reduction}, we can assume that there are
    no distinct vertices $u$ and $v$ in $S$ such that $N(u)\subseteq N(v)$, that
    is, the family of neighborhoods forms an antichain.

    Observe that there is no vertex in $S$ with degree 3. Proceeding by
    contradiction, suppose that there is $x\in S$ with $d_G(x)=3$. Since
    $|S|\geq 2$, $N(S)=K$, and the family of neighborhoods is an antichain, for
    any $y \in S \setminus \{x\}$ the subgraph induced by $K\cup\{x,y\}$ in $G$
    is isomorphic to $G_{i}$, for some $i\in \{1,2,3\}$, which is impossible. 

    Let $S_1$ be the subset of vertices of $S$ with degree 1. We consider the
    following possibilities.

    \noindent\textbf{Case 1}. There are two distinct vertices $x_0,x_1\in S_1$.
    Since the family of neighborhoods is an antichain, we can assume that
    $0x_0,1x_1\in E(G)$, and no other vertex in $S$ is adjacent to $0$ nor $1$
    in $G$. Notice that, $1\leq |S\setminus \{x_0,x_1\}|\leq 2$. If $|S\setminus
    \{x_0,x_1\}|=2$, then $S=S_1$ and we can assume without loss of generality
    that $S=\{x_0,x_1,x_2,x_3\}$ and $x_i$ is adjacent to $i$ in $G$, for every
    $i\in \{0,1,2,3\}$. We define the following coloring $c$: for every $i\in \{
    0,1,2,3\}$, set $c(i)=i$ and $c(x_i)=i+1$, addition taken modulo $4$. It is
    readily seen that $c$ is a star $4$-coloring of $G$. Suppose now that
    $|S\setminus \{x_0,x_1\}|=1$, and let $x_2$ be the unique vertex in
    $S\setminus \{x_0,x_1\}$. It follows that $N(x_2)=\{2,3\}$. We define the
    following coloring $c$: $c(x_0)=2$, $c(x_1)=0$, $c(x_3)=1$ and, for every
    $i\in \{0,1,2,3\}$, $c(i)=i$. It is straightforward to prove that $c$ is a
    star $4$-coloring. See Figure~\ref{fig: case 1 split_mo_K=4}.

    \noindent\textbf{Case 2}. $|S_1|=1$. Let $x_0$ be the unique vertex in
    $S_1$. Suppose, without loss of generality, that $x_0$ is adjacent to $0$.
    Since the family of neighborhoods is an antichain, and $N(S)=K$, we have
    that $2\leq |S\setminus\{x_0\}|\leq 3$, otherwise $G$ is isomorphic to
    $G_1$. Moreover, if $|S\setminus\{x_0\}|=3$, then $G$ is isomorphic to
    $G_{5}$. Suppose that $S\setminus \{x_0\}=\{x_1,x_2\}$. Since the family of
    neighborhoods is an antichain, and there is no vertex of degree 3, we can
    assume that $N(x_1)=\{1,2\}$, and $N(x_2)=\{2,3\}$. By coloring the vertices
    of $G$ as in Figure~\ref{fig: case 2 split_mo_K=4}, we obtain a star
    $4$-coloring of $G$.

    \noindent\textbf{Case 3}. $S_1=\varnothing$. It implies that every vertex in
    $S$ has degree 2. Note that for every vertex $x\in S$, if there is $y\in S$
    such that $N(y)=K\setminus N(x)$, then the subgraph of $G$ induced by
    $K\cup\{x,y\}$ is isomorphic to $G_{4}$, which is impossible. Hence, for
    every vertex $x\in S$, there is no vertex $y\in S$ such that
    $N(y)=K\setminus N(x)$. Moreover, since $N(S)=K$, and the family of
    neighborhoods is an antichain, we obtain that $|S|=3$. Therefore, there is
    only a single possible configuration, to which a star $4$-coloring can be
    assigned, as shown in Figure~\ref{fig: case 3 split_mo_K=4}.
\end{proof}

\begin{figure}[htb!]
    \centering
    \begin{subfigure}[b]{0.42\linewidth}
        \centering
            \includegraphics[width=\textwidth]{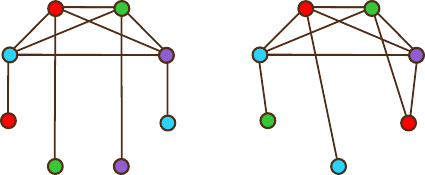}
            \caption{Case 1} 
        \label{fig: case 1 split_mo_K=4}
    \end{subfigure}
    \hfill
        \begin{subfigure}[b]{0.17\linewidth}
        \centering
            \includegraphics[width=\textwidth]{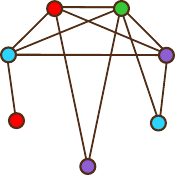}
            \caption{Case 2} 
        \label{fig: case 2 split_mo_K=4}
    \end{subfigure}
    \hfill
        \begin{subfigure}[b]{0.17\linewidth}
        \centering
            \includegraphics[width=\textwidth]{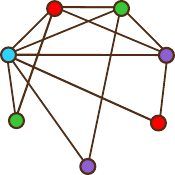}
            \caption{Case 3}
        \label{fig: case 3 split_mo_K=4}
    \end{subfigure}
        \caption{Cases of Proposition~\ref{pro: star $4$-col iff F_4-free}.}
        \label{fig: cases split_mo_K=4}
\end{figure}

As usual, the knowledge of the complete set of minimal obstructions to a split
graph to be star $4$-colorable can be used to find a certifying algorithm to
test whether an input split graph has star chromatic number at most $4$. Notice
that similar algorithms to that in Corollary~\ref{cor:split-4-alg} can be easily
stated for star $2$- and $3$-coloring of split graphs.

\begin{corollary}
\label{cor:split-4-alg}
    There is a certifying algorithm that runs in time $O(|V|+|E|)$ to test
    whether an input split graph $G$ has a star $4$-coloring.
\end{corollary}
\begin{proof}
    We can order the vertices of $G$ according to their degrees in time
    $O(|V|)$, which in turn enables us to find a split partition $(K,S)$ of $G$
    that maximizes the cardinality of $K$ within the same time bound. If $K$ has
    more than $4$ vertices, we return a copy of $K_5$ as a no-certificate. If
    $K$ has fewer than 4 vertices, we assign a different color to each vertex in
    $K$, and a single additional color to every vertex in $S$.   Otherwise, set
    $K=\{0,1,2,3\}$, and use $O(|V|+|E|)$ time to classify each vertex in $S$
    according to its neighborhood, this is, each vertex in $S$ receives a label
    from the power set $\mathcal{P}(\{0, 1, 2, 3\})$. In time $O(1)$ we choose
    the maximal labels among the ones used and proceed to perform the following
    checks:
    \begin{itemize}
        \item If there is an element of $\{0, 1, 2, 3\}$ which is not used
            by any label, we color all vertices in $S$ with such an element
            and return the resulting coloring as a yes-certificate.
        \item If there are two labels such that their union is the set $\{
            0, 1, 2, 3\}$, we return two vertices in $S$ having those labels,
            together with $K$, as a no-certificate.
        \item If there are two labels of cardinality $1$, we follow the first
            case of Proposition~\ref{pro: star $4$-col iff F_4-free} to
            construct a star $4$-coloring of $G$ and return it as a
            yes-certificate.
        \item If there is only one label of cardinality $1$, we follow the
            second case of Proposition~\ref{pro: star $4$-col iff F_4-free}
            to either construct a star $4$-coloring of $G$, which is returned
            as a yes-certificate, or a copy of $G_1$ or $G_5$ in
            Figure~\ref{fig:split_mo_K=4}, which is returned as a
            no-certificate.
        \item If there are no labels of cardinality $1$, we follow the third
            case of Proposition~\ref{pro: star $4$-col iff F_4-free} to either
            construct a star $4$-coloring of $G$, which is returned as a
            yes-certificate, or a copy of $G_4$ in
            Figure~\ref{fig:split_mo_K=4}, which is returned as a
            no-certificate.
    \end{itemize}
    Since the number of different labels is constant, all the aforementioned
    tests can be performed in time $O(1)$ for the chosen labels, but we may
    need additional $O(|V|)$ time to color all the vertices of $G$ when a
    star $4$-coloring exists.
\end{proof}

We now turn our attention to split graphs admitting a star $5$-coloring.  Our
treatment of this case is very similar to the previous one, albeit more
complicated.   On the bright side, we obtain a similar algorithmic result that
let us propose a conjecture regarding the star coloring of split graphs.

Set the family $\mathcal{F}_5^{\mathrm{sp}} = \{ K_6\}\cup\{ J_i\colon\ 1\leq
i\leq 12\}$ depicted in Figure~\ref{fig:split_mo_K=5}. We first establish that
these graphs are minimal obstructions, and then show that they characterize star
$5$-colorability in split graphs.

\begin{figure}[htb!]
    \centering
    \includegraphics[width=1\textwidth]{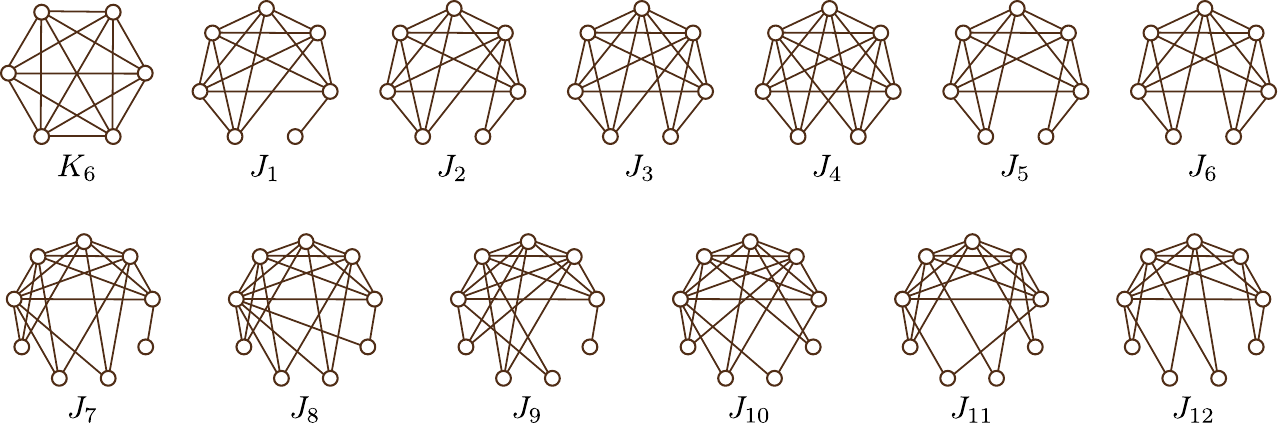}
    \caption{Minimal obstructions to star $5$-coloring in split graphs.} 
    \label{fig:split_mo_K=5}
\end{figure}

\begin{lemma}
\label{lemma:split-5-minimal}
    Every graph in $\mathcal{F}_5^{\mathrm{sp}}$ is a minimal obstruction to
    star $5$-colorability. 
\end{lemma}

\begin{proof}
    For every $i\in \{1, 2, 3, 5, 6, 9, 12\}$, the graph $J_i$ can be obtained
    by adding a true twin to $G_1$, $G_2$, $G_3$, or $G_5$, whereas $K_6$,
    $J_4$, and $J_8$ can be obtained by adding a universal vertex to $K_5$,
    $G_3$, and $G_5$, respectively.  By Propositions~\ref{pro:add true twins}
    and \ref{pro: split add universal vertex}, $K_6$ and $J_i$ are minimal
    obstructions to star $5$-colorability, for every $i\in \{1, 2, 3, 4, 5, 6,
    8, 9, 12\}$.  Moreover, by Proposition~\ref{pro:omega-obstructions}, $J_7$,
    $J_{10}$, and $J_{11}$ are obstructions to star $5$-colorability.  Note
    that, for every $i\in \{7, 10, 11\}$, if some vertex $v$ in $K$ is removed
    from $J_i$, then by Proposition~\ref{pro: split omega or omega+1},
    $J_i-v$ is star $5$-colorable.  Additionally, Figure~\ref{fig:J7 J10 J11
    mo split K=5} depicts a star $5$-coloring for every vertex-deleted subgraph
    of $J_i$ when the deleted vertex belongs to $S$.  Hence, $J_i$ is a minimal
    obstruction to star $5$-colorability, for every $i \in \{7, 10, 11\}$.
    Therefore, every graph in $\mathcal{F}_5^{\mathrm{sp}}$ is a minimal
    obstruction to star $5$-colorability.
\end{proof}

\begin{figure}[t!]
    \centering
    \includegraphics[width=0.82\linewidth]{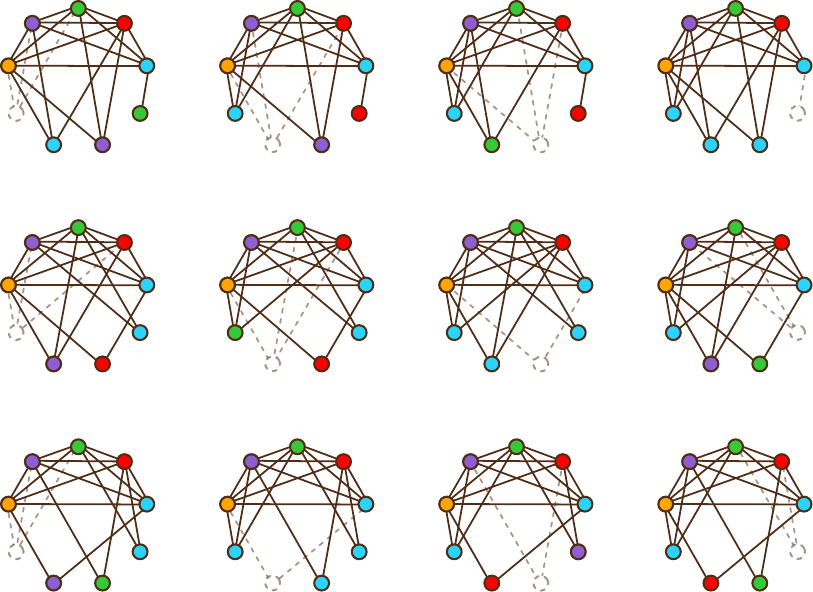}
    \caption{Star $5$-colorings of each subgraph of $J_7$, $J_{10}$ or $J_{11}$
    obtained by deleting a vertex of the set $S$.  From top to bottom, the rows
    correspond to $J_7$, $J_{10}$ and $J_{11}$, respectively; in each drawing
    the grey vertex (and its incident grey edges) indicates the deleted vertex.}
\label{fig:J7 J10 J11 mo split K=5}
\end{figure}

\begin{proposition}
\label{pro: star $5$-col iff F_5-free}
    If $G$ is a split graph, then $G$ is star $5$-colorable if and only if $G$
    is $\mathcal{F}_5^{\mathrm{sp}}$-free. 
\end{proposition}

\begin{proof}
    By Lemma~\ref{lemma:split-5-minimal}, every star $5$-colorable graph is
    $\mathcal{F}_5^{\mathrm{sp}}$-free. Conversely, suppose that $G$ is a
    $\mathcal{F}_5^{\mathrm{sp}}$-free graph with split partition $(K, S)$. To
    prove that $G$ is star $5$-colorable, it is enough to check that each of its
    connected components is, so we can assume that $G$ is connected.
    
    Since $K_6 \in \mathcal{F}_5^{\mathrm{sp}}$, it follows that $|K| \leq 5$
    and each vertex $x\in S$ satisfies $1\leq d_G (x)\leq 4$.   By
    Proposition~\ref{pro: split omega or omega+1}, if $|K|\leq 4$, then $G$ is
    star $5$-colorable.   Hence, assume that $|K|=5$, and set $K=\{0,1,2,3,4\}$.
    By Proposition~\ref{pro:split-K-insaturated} we may assume $N(S)=K$, so
    $|S|\ge 2$.   Moreover, by Proposition~\ref{pro:neighborhood-reduction}, we
    may assume that there are no distinct vertices $u,v\in S$ such that
    $N(u)\subseteq N(v)$, that is, the family of neighborhoods forms an
    antichain; we refer to this as \emph{the antichain condition} on $S$.   

    Observe that, since $|S|\geq 2$, $K=N(S)$, and $S$ satisfies the antichain
    condition, if there is $x\in S$ with $d_G(x)=4$, then for any other vertex
    $y$ in $S$, the subgraph induced by $K\cup\{x,y\}$  is isomorphic to
    $J_{i}\in \mathcal{F}_5^{\mathrm{sp}}$, for some $i \in \{1,2,3,4\}$, a
    contradiction.   Hence, for every vertex $x\in S$ we have $1\leq d_G(x)\leq
    3$. Let  $S_i$ be the set of vertices in $S$ with degree $i$ in $G$. 

    We now show that: 
    \begin{equation}
    \label{eq:remark-S3}
        \textrm{for every pair of distinct vertices $x, y \in S_3$, we have
        $|N(x)\cap N(y)|=2$.}
    \end{equation}
    Indeed, as $|K|= 5$ and $S$ satisfies the antichain condition, we have that
    $1\leq |N(x)\cap N(y)|\leq 2$. Moreover, if $|N(x)\cap N(y)|=1$, then $K\cup
    \{x,y\}$ induces a subgraph isomorphic to $J_6$, which is impossible. Thus,
    property~\eqref{eq:remark-S3} holds.  We distinguish some different cases. 

    \medskip
    \noindent\textbf{Case 1.} Suppose $N(S_3)=K$.   Clearly, $S_1=\varnothing$
    because $S$ satisfies the antichain condition.   Moreover, by
    property~\eqref{eq:remark-S3} we have $|S_3|\ge 3$.   First, suppose that
    $|S_3| \ge 4$, and let $x_1, x_2, x_3, x_4\in S_3$. Assume without loss of
    generality that $N(x_1) = \{0,1,2\}$ and $N(x_2) = \{0,1,3\}$.  It is
    straightforward to check that, without loss of generality, either $N(x_3) =
    \{0,1,4\}$ and $N(x_4) = \{1,2,3\}$, or $N(x_3) = \{1,2,3\}$ and $N(x_4) =
    \{0,1,4\}$.   However, in both cases we have $|N(x_3)\cap N(x_4)| = 1$,
    contradicting the property \eqref{eq:remark-S3}.   Thus, it must be the case
    that $|S_3|=3$, and since $N(S_3) = K$, we may relabel $S_3 = \{x_1, x_2,
    x_3\}$ in such a way that $N(x_1) = \{0, 1, 2\}$, $N(x_2) = \{0, 1, 3\}$,
    and $N(x_3) = \{0, 1, 4\}$.   Therefore, since $S$ satisfies the antichain
    condition, and $G$ is a $J_5$-free graph, we have that $S_2 = \varnothing$,
    and hence $S = \{x_1, x_2, x_3\}$.  In this case $G$ is star $5$-colorable,
    as can be checked in Figure~\ref{fig: case 1 split_K=5}.

    \begin{figure}[htb!]
     \centering
     \begin{subfigure}[b]{0.3\textwidth}
         \centering
             \includegraphics[width=0.6\textwidth]{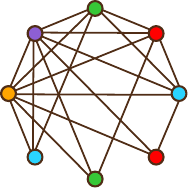}
             \caption{Case 1} 
         \label{fig: case 1 split_K=5}
     \end{subfigure}
     \begin{subfigure}[b]{0.3\textwidth}
         \centering
             \includegraphics[width=0.6\textwidth]{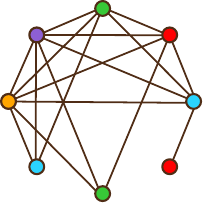}
             \caption{Case 2.1} 
         \label{fig: case 21 split_K=5}
     \end{subfigure}
     \begin{subfigure}[b]{0.3\textwidth}
         \centering
             \includegraphics[width=0.6\textwidth]{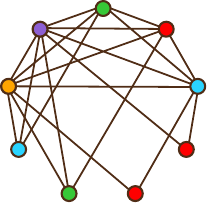}
             \caption{Case 2.2}
         \label{fig: case 22 split_K=5}
     \end{subfigure}

    \caption{Cases 1 and 2 of Proposition~\ref{pro: star $5$-col iff
    F_5-free}.}
    \label{fig: cases 1 and 2 split_K=5}
    \end{figure}

    \medskip
    \noindent\textbf{Case 2.} Suppose $|N(S_3)|=4$.   Without loss of generality
    assume that $K \setminus N(S_3) = \{4\}$, and let $y\in S\setminus S_3$ be
    such that $4 \in N(y)$, so $y\in S_1\cup S_2$.   Let $x_1,x_2\in S_3$. By
    property~\eqref{eq:remark-S3} we can assume without loss of generality that
    \[N(x_1)=\{0,1,2\} \ \text{ and }\ N(x_2)=\{0,1,3\}.\] We claim that
    $|S_3|=2$.   Otherwise, let $x_3$ be another vertex in $S_3$. Since $S$
    satisfies the antichain condition, $N(x_1)=\{0,1,2\}$, $N(x_2)=\{0,1,3\}$,
    and $|N(s)| = 4$, we must have $2,3 \in N(x_3)$.   Moreover, by symmetry of
    vertices $0$ and $1$, we may assume $N(x_3)=\{0,2,3\}$.   If $y\in S_1$,
    then $G[K\cup \{x_1,x_2,x_3,y\}]\cong J_7$, which is impossible.   Thus,
    $y\in S_2$.  If there is $j \in \{1,2,3\} \cap N(y)$, then $G[K \cup \{y,
    x_{4-j}\}] \cong J_5$, but we are assuming $G$ is $J_5$-free. Hence
    $N(y)=\{0,4\}$, but then $G[K\cup \{x_1,x_2,x_3,y\}]\cong J_8$, an absurd.
    Therefore, $|S_3|=2$.
    
    \begin{itemize}
        \item[\textbf{2.1.}] Suppose $y\in S_1$. As $S$ satisfies the antichain
        condition, we have $S_1=\{y\}$ and $|S_2| \le 1$. If $|S_2| = 1$, let us
        say $S_2 = \{w\}$, then $N(w) = \{2,3\}$; however, in this case $G[K\cup
        \{x_1,x_2,y,w\}]\cong J_9$, a contradiction. Assume therefore that $S_2
        = \varnothing$, and assign to $G$ the star $5$-coloring depicted in
        Figure~\ref{fig: case 21 split_K=5}. 

        \item[\textbf{2.2.}] Suppose $y\in S_2$.   Due to the antichain
        condition, we have $S_1 = \varnothing$.   Moreover, since $G$ is a
        $J_5$-free graph, we must have $\{2,3\} \cap N(y) = \varnothing$. Notice
        that $|S_2|\le 2$, because the only possible neighborhoods for vertices
        in $S_2$ are $\{0,4\}$ and $\{1,4\}$.   A star $5$-coloring of $G$ when
        $|S_2|=2$ is shown in Figure~\ref{fig: case 22 split_K=5}; a star
        $5$-coloring of $G$ for the case $|S_2|=1$ can be obtained as a
        restriction of this same coloring.
    \end{itemize}

    \begin{figure}[htb!]
    \centering
     \begin{subfigure}[b]{0.3\textwidth}
         \centering
             \includegraphics[width=0.6\textwidth]{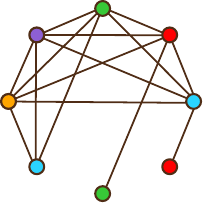}
             \caption{Case 3.1} 
         \label{fig: case 31 split_K=5}
     \end{subfigure}
     \begin{subfigure}[b]{0.3\textwidth}
         \centering
             \includegraphics[width=0.6\textwidth]{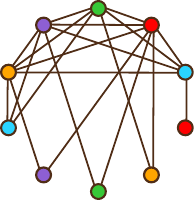}
             \caption{Case 3.2} 
         \label{fig: case 32 split_K=5}
     \end{subfigure}
     \begin{subfigure}[b]{0.3\textwidth}
         \centering
             \includegraphics[width=0.6\textwidth]{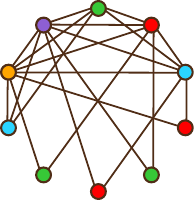}
             \caption{Case 3.3}
         \label{fig: case 33 split_K=5}
     \end{subfigure}
        \caption{Case 3 of Proposition~\ref{pro: star $5$-col iff F_5-free}.}
        \label{fig: cases 3 split_K=5}
    \end{figure}
    
    \noindent\textbf{Case 3}. Suppose $|N(S_3)|=3$.   We have $|S_3|=1$ because
    $S$ satisfies the antichain condition.   Let $x$ be the only vertex in
    $S_3$, and assume, without loss of generality, that $N(x)=\{ 0, 1, 2\}$.
    Notice that no vertex in $S$ is adjacent to both $3$ and $4$, otherwise $G$
    would have $J_5$ as an induced subgraph.   Let $w,y\in S$ be such that $3\in
    N(w)$ and $4\in N(y)$.   
    \begin{itemize}
        \item[\textbf{3.1.}] If $y,w\in S_1$, then $S=\{x,y,w\}$.   A star
        $5$-coloring of $G$ is depicted in
        Figure~\ref{fig: case 31 split_K=5}.   
        
        \item[\textbf{3.2.}] Suppose that exactly one of $y$ and $w$ belongs to
        $S_1$. Without loss of generality, assume $S_1 = \{y\}$, and $w\in S_2$.
        Since $S$ satisfies the antichain condition, any vertex $z\in S_2$ must
        satisfy $N(z)=\{i,3\}$, for some $i\in \{0,1,2\}$. Figure~\ref{fig: case
        32 split_K=5} depicts a star $5$-coloring of $G$ when $|S_2|=3$; the
        cases when $|S_2|\le 2$ are easy consequences of this case.

        \item[\textbf{3.3.}] If cases 3.1 and 3.2 do not occur, thus
        $S_1=\varnothing$ and $y,w\in S_2$. Without loss of generality, assume
        that $N(w)=\{0,3\}$ and $N(y)=\{1,4\}$. If there exists a vertex $z\in
        S_2$ such that $2\in N(z)$, then $G[K\cup\{x,w,y,z\}]\cong J_{11}$, a
        contradiction. Thus no vertex of $S_2$ is adjacent to $2$. By the
        antichain condition, it follows that $|S_2|\le 4$, and every vertex
        $z\in S_2\setminus\{w,y\}$ satisfies either $N(z)=\{0,4\}$ or
        $N(z)=\{1,3\}$. A star $5$-coloring of $G$ when $|S_2|=4$ is depicted in
        Figure~\ref{fig: case 33 split_K=5}; the cases $|S_2|\le 3$ are obtained
        by restricting this coloring.
    \end{itemize}
    
    \begin{figure}[t!]
        \centering
        \includegraphics[width=0.8\linewidth]{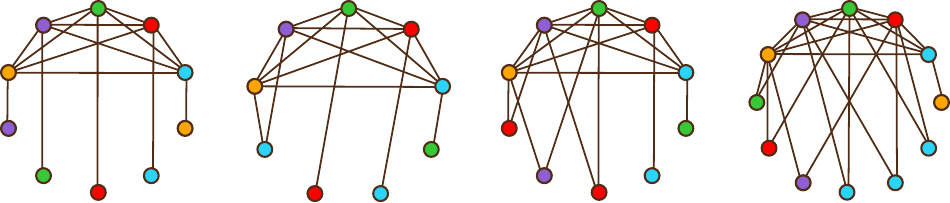}
        \caption{Case 4 when $S_1\neq \varnothing$ of Proposition~\ref{pro:
        star $5$-col iff F_5-free}.}
        \label{fig: case 4 S_1 non empty split_K=5}
    \end{figure}

    \noindent\textbf{Case 4.} Suppose $S_3=\varnothing$, so $S=S_1\cup S_2$.
    Since $S$ satisfies the antichain condition, for each $z\in S_1$ no other
    vertex in $S$ is adjacent to the only neighbor of $z$.   Figure~\ref{fig:
    case 4 S_1 non empty split_K=5} depicts the four possible cases when
    $S_1\ne\varnothing$, together with the corresponding star $5$-coloring; for
    each of these cases, the maximum number of vertices in $S_2$ is considered.
    
    Now, assume that $S_1= \varnothing$, so $S = S_2$. Thus, each vertex $s \in
    S$ has a 2-subset of $\{0, 1, 2, 3, 4, 5\}$ as neighborhood; let us denote
    by $x_{ij}$ a vertex in $S$ whose neighborhood is $\{i,j\}$.   Fix $i\in K$
    such that $|N(i)\cap S|$ is maximum.   We can assume without loss of
    generality that $i=0$, so $2 \le |N(0)\cap S| \le 4$.   We distinguish the
    following cases.   

    \begin{itemize}
        \item[\textbf{4.1.}] Suppose $|N(0)\cap S| \ge 3$.   If there exist
        $x_{pq}, x_{\ell t} \in S$ such that $\{p, q, \ell, t\} = K \setminus
        \{0\}$, then either $x_{0p}, x_{0q} \in S$ or $x_{0\ell}, x_{0t} \in S$.
        By symmetry, we can assume the former case happens, but then $G[K \cup
        \{x_{pq}, x_{\ell t}, x_{0p}, x_{0q}\}] \cong J_{12}$, which is
        impossible.   

        Thus, we may assume that all vertices in $S$ not adjacent to $0$, are
        adjacent to a same vertex in $K$, say $1$.   Hence $S\subseteq
        \{x_{01},x_{02},x_{03},x_{04},x_{12},x_{13},x_{14}\}$. Figure~\ref{fig:
        case 41 split_K=5} depicts a star $5$-coloring for the case when the
        equality holds; any other cases follows easily from this one.

        \item[\textbf{4.2.}] Suppose $|N(0)\cap S| = 2$.   Observe that if there
        are pairwise distinct vertices $p, q, r \in K$ such that $x_{pq},
        x_{pr}, x_{qr} \in S$, then we also have $x_{\ell t} \in S$, where
        $\{\ell,t\} = K \setminus \{p, q, r\}$.   Nevertheless, in such a case
        $G[K \cup \{x_{pq}, x_{pr}, x_{qr}, x_{\ell t}\}] \cong J_{12}$, a
        contradiction. Therefore, up to relabeling the vertices in $K$, we may
        assume 
        \(
            S \subseteq \{x_{01},x_{12},x_{23},x_{34},x_{40}\}.
        \)
        See Figure~\ref{fig: case 42 split_K=5} for a star $5$-coloring for the
        case when equality holds; any other cases follows easily from this one.
    \end{itemize}

    \begin{figure}[htb!]
     \centering
     \begin{subfigure}[b]{0.35\textwidth}
         \centering
             \includegraphics[width=0.58\textwidth]{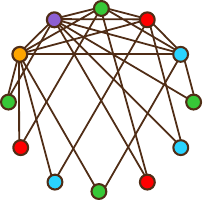}
             \caption{Case 4.1} 
         \label{fig: case 41 split_K=5}
     \end{subfigure}
          \begin{subfigure}[b]{0.35\textwidth}
         \centering
             \includegraphics[width=0.55\textwidth]{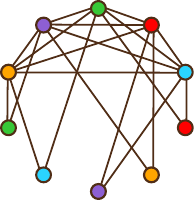}
             \caption{Case 4.2}
         \label{fig: case 42 split_K=5}
     \end{subfigure}
        \caption{Cases 4.1 and 4.2 of Proposition~\ref{pro: star $5$-col
        iff F_5-free}.}
        \label{fig: cases 4 split_K=5}
    \end{figure}
\end{proof}

As with star $4$-colorings, we are now able to propose a certifying recognition
algorithm for split graphs having star chromatic number at most $5$.

\begin{corollary}
\label{cor:split-5-alg}
    There is a certifying algorithm that runs in time $O(|V|+|E|)$ to test
    whether an input split graph $G$ has a star $5$-coloring.
\end{corollary}

\begin{proof}[Proof sketch.]
    The proof is analogous to that of Corollary~\ref{cor:split-4-alg}.  The idea
    is to obtain a split partition $(K,S)$ of the input graph $G$ that maximizes
    the cardinality of $K$.  If $|K| \ge 6$, a $K_6$ is returned as a
    no-certificate.   If $|K| \le 4$, then every vertex in $S$ receives a fifth
    color and the resulting coloring is returned as a yes-certificate. When $|K|
    = 5$, every vertex in $S$ is explored and assigned a label from the power
    set $\mathcal{P}(\{ 0, 1, 2, 3, 4 \})$.   Afterwards, the set of used labels
    is explored to extract the maximal ones, and the proof of
    Proposition~\ref{pro: star $5$-col iff F_5-free} is used as a guide to
    explore the different possible cases for the set of labels, yielding either
    a no-certificate or a yes-certificate in every case.

    Since the number of different possible labels is constant with respect to
    the order of the input graph $G$, and only a subset of the labels is
    explored, the running time is again $O(|V|+|E|)$.
\end{proof}

Let $G = (K, S)$ be a split minimal obstruction to star $6$-coloring that is not
isomorphic to $K_7$. It follows from Proposition~\ref{pro: split omega or
omega+1} that $G$ has clique number $6$, and by
Proposition~\ref{pro:split-K-insaturated}, we have $N(S) = K$. Moreover, by
Proposition~\ref{pro: split omega or omega+1}, the neighborhoods of distinct
vertices in $S$ are incomparable; in particular, it follows from Sperner’s
Theorem that $|S| \le 20$, and therefore $G$ has at most $26$ vertices.

With the help of a computer and the software \texttt{nauty} \cite{MCKAY201494},
we obtained the complete list of all non-isomorphic split graphs satisfying
these constraints and determined which of them are minimal obstructions to star
$6$-colorability. There are $68$ such obstructions, which we list in
Appendix~\ref{app:smoS6c}.

By looking at the proofs of Propositions~\ref{pro: star $4$-col iff F_4-free}
and \ref{pro: star $5$-col iff F_5-free}, it is not hard to imagine how a proof
for any fixed value of $k$ could be obtained.   Of course, the length of such a
proof even for the case $k = 6$ is too long to be considered practical.
Nonetheless, if we believe that such a proof exists, then a certifying algorithm
similar to the one described in Corollaries \ref{cor:split-4-alg} and
\ref{cor:split-5-alg} also exists.   This inspires us to propose the following
conjecture.

\begin{conjecture}
\label{con:linear-split}
    Let $k$ be a fixed positive integer $k$.   There is a certifying algorithm
    that runs in time $O(|V|+|E|)$ to solve the problem of determining whether
    an input split graph $G$ admits a star $k$-coloring.
\end{conjecture}

\section{2-trees}
\label{sec:2-trees}

A graph $G$ is a \textit{$2$-tree} if it can be constructed starting from a
triangle and repeatedly adding new vertices of degree~$2$, each time making the
new vertex adjacent to the two endpoints of an existing edge.  Equivalently, $G$
is a chordal graph in which every maximal clique has size~$3$ and every minimal
vertex separator has size~$2$. 

A \textit{$2$-path} is a $2$-tree whose maximal cliques (the triangles) can be
ordered as $T_1,\dots,T_m$ in such a way that each $T_i$ with $i\ge 2$ shares
exactly one edge with $T_{i-1}$ and has no vertices in common with $T_j$ for
$j<i-1$. In other words, the clique graph of $G$ is a path. We refer to such an
ordering as a \textit{clique path} of $G$. Every $2$-path has exactly two
simplicial vertices (i.e., vertices whose neighbors induce a clique).  

\subsection{A canonical construction of a 2-path}

Let $G$ be a $2$-path. Then $G$ has precisely two simplicial vertices; we denote
them by $x_1$ and $x_2$. Fix one of them, say $x_1$, and let $T_1$ be the unique
triangle of $G$ containing $x_1$. Starting from $T_1$, the $2$-path $G$ can be
obtained by repeatedly adding a new vertex of degree~$2$, each time adjacent to
the two endpoints of one of the two edges containing the last added vertex. This
yields a canonical construction of $G$, unique up to reversal (the other
direction being obtained by starting from the triangle containing $x_2$). 

More precisely, we obtain an ordered sequence of triangles
\[
  T_1,T_2,\dots,T_m
\]
such that for each $i\in\{1,\dots,m-1\}$ the triangle $T_{i+1}$ is obtained by
adding a new vertex that is adjacent to both endpoints of an edge of $T_i$ that
is not an edge of $T_{i-1}$. For each $i\in\{1,\dots,m-1\}$ we denote by $e_i$
the edge of $T_i$ that serves as the base for attaching the triangle $T_{i+1}$,
and we refer to $e_1,\dots,e_{m-1}$ as the \textit{base edges} of this
construction. 

The base edges naturally form a tree.   Since in every step we add a new
triangle along an unused edge in the last triangle, inductively we observe that
we are adding a leaf to a tree.   \textit{The base caterpillar} of the graph $G$
is the subgraph $B_G$ of $G$ induced by the base edges of a canonical triangle
construction of $G$.   As its name suggests, $B_G$ is actually a caterpillar.

\begin{proposition}
\label{pro:BG-caterpillar}
    The graph $B_G$ is a caterpillar.
\end{proposition}
\begin{proof}
    We proceed by induction on $|G|$. If $|G|=4$, then $G$ is a diamond and
    $B_G$ consists of the unique base edge, that is, the chord of the diamond.
    In particular, $B_G$ is a caterpillar (a single edge).

    Now assume $|G|>4$ and let $v$ be a simplicial vertex of $G$. Let $G'$ be
    the graph $G' = G-v$. Then $G'$ is again a $2$-path  with $|G'|=|G|-1$. By
    the induction hypothesis, $B_{G'}$ is a caterpillar. Let $e=xy$ be the
    unique edge of $G'$ such that $v$ is adjacent to both $x$ and $y$;
    equivalently, $\{v,x,y\}$ induces the triangle in which $v$ is introduced,
    and $e$ is the corresponding base edge. Then $B_G$ is obtained from $B_{G'}$
    by adding the edge $e$. Moreover, one of the endpoints of $e$ lies on the
    last base edge of the canonical construction of $G'$, and hence attaching
    $e$ to $B_{G'}$ preserves the caterpillar structure. Therefore $B_G$ is a
    caterpillar.
\end{proof}

We encode the structure of $G$ by a sequence of integers derived from $B_G$. Let
$b_1,\dots,b_r$ be the non-leaf vertices of $B_G$, listed in the order in which
they appear along the spine of $B_G$, according to the canonical construction
above. For each $j\in\{1,\dots,r\}$, let $d_j$ be the number of leaves of $B_G$
adjacent to $b_j$, in other words
\[
  d_j = d_{B_G}(b_j). 
\]
The \textit{degree sequence associated with $G$}, denoted $\sigma(G)$, is
defined to be
\[
  \sigma(G)=(d_1,\dots,d_r).
\]
This sequence records the degrees of the non-leaf vertices of the base
caterpillar, in the canonical order induced by the construction starting at
$x_1$. Up to reversal, this sequence does not depend on the choice of simplicial
vertex from which we start, and it provides a convenient way to describe the
structure of a $2$-path in our arguments.

We now show that the degree sequence associated with a $2$-path determines the
graph uniquely (up to isomorphism and reversal of the sequence). 

\begin{proposition}
\label{pro:2path-sequence}
    Let $G$ be a $2$-path and let $\sigma(G)$ be its associated degree sequence.
    We have:
    \begin{enumerate}[label=(\alph*)]
        \item the sequence $\sigma(G)$ is well-defined up to reversal, that is,
            starting the canonical construction from $x_1$ or from $x_2$
            produces sequences that are reverses of each other; and,
        \item if $G$ and $G'$ are $2$-paths such that $\sigma(G)=\sigma(G')$ or
            $\sigma(G)$ is the reversal of $\sigma(G')$, then $G$ and $G'$ are
            isomorphic.
    \end{enumerate}
    In particular, the map that sends a $2$-path $G$ (up to isomorphism) to its
    associated degree sequence $\sigma(G)$ (up to reversal) is a bijection
    between $2$-paths and finite integer sequences with entries at least $2$,
    where the empty sequence corresponds to the diamond.
\end{proposition}
\begin{proof}
    We first prove $(a)$. Let $x_1$ and $x_2$ be the two simplicial vertices of
    $G$. Starting the canonical construction from $x_1$ produces an ordered
    sequence of base edges
    \[
    e_1,e_2,\dots,e_{m-1}.
    \]
    Starting from $x_2$ produces the same base edges, but in the opposite order.
    Indeed, at each step of a canonical construction, the next triangle is
    forced: one deletes the current simplicial vertex and moves to the unique
    triangle that contains it. Thus the two possible canonical constructions
    traverse the same sequence of triangles in opposite directions.
    Consequently, they induce the same base caterpillar $B_G$, with its spine
    traversed in opposite directions. Therefore the corresponding degree
    sequences are reverses of one another. This proves that $\sigma(G)$ is
    well-defined up to reversal.

    We now prove $(b)$. We show that the sequence $\sigma(G)$ determines $G$ up
    to isomorphism and reversal. Let
    \[
    \sigma(G)=(d_1,\dots,d_r).
    \]
    The sequence determines the base caterpillar $B_G$ uniquely up to
    isomorphism. Indeed, its spine has vertices
    \[
    b_1,b_2,\dots,b_r
    \]
    in this order, and the number of leaves attached to each $b_j$ is determined
    by $d_j$: if $r=1$, then $b_1$ has $d_1$ leaves; if $r\ge 2$, then the
    endvertices $b_1$ and $b_r$ have respectively $d_1-1$ and $d_r-1$ leaves,
    while each internal spine vertex $b_j$, with $1<j<r$, has $d_j-2$ leaves.
    Thus the degree sequence determines the caterpillar $B_G$, up to reversing
    the order of the spine.

    It remains to observe that a $2$-path is uniquely determined by its base
    caterpillar together with the choice of one end of the canonical
    construction. Indeed, once $B_G$ is fixed, the canonical construction is
    forced: consecutive base edges determine consecutive triangles, and at each
    step the new vertex is attached to the endpoints of the current base edge.
    Equivalently, for each vertex $b_j$, the neighbors of $b_j$ in $B_G$ appear
    consecutively as the rim of the unique maximal fan centered at $b_j$, and
    the overlap between consecutive fans is forced by the order of the spine.
    Hence, two $2$-paths with the same associated degree sequence have
    isomorphic base caterpillars and the same forced fan overlaps, and therefore
    are isomorphic.

    If $\sigma(G)$ is the reversal of $\sigma(G')$, then the same argument
    applies after reversing the canonical construction of one of the graphs.
    Hence $G\cong G'$ in this case as well. This proves $(b)$.

    For the last statement, observe that the empty sequence corresponds to the
    diamond, and every non-empty finite sequence of integers at least $2$
    determines, by the construction above, a unique $2$-path up to isomorphism
    and reversal. 
\end{proof}

We next relate degrees in $B_G$ and $G$. Consider a canonical triangle
construction $T_1,\dots,T_r$ of $G$. Whenever a vertex $u$ is introduced, it is
attached to a base edge $e=xy$, so $V(T_i)=\{u,x,y\}$ with $ux,uy\in E(G)$. In
the next step, the construction continues through exactly one of the edges $ux$
or $uy$ (say $ux$), which becomes the next base edge. Hence, among the two edges
of $T_i$ incident to $u$, exactly one (the continuation edge $ux$) is recorded
in $B_G$, whereas the edge $xy$ joining the two older vertices belongs to $G$
but does not appear in $B_G$. This accounts for one additional neighbor of $u$
in $G$ not seen in $B_G$ in this direction. Since a $2$-path admits two
canonical constructions (one starting from each simplicial vertex), we obtain
one such additional neighbor at each end. Therefore, every $u\in V(B_G)$ has
exactly two neighbors in $G$ that are not its neighbors in $B_G$, and so
\[
    d_G(u)=d_{B_G}(u)+2.
\]

Let $\spn$ be the \textit{spine} of $B_G$, that is, the central path of $B_G$
obtained after deleting all leaves. Write $V(\spn)=\{u_1,u_2,\dots,u_t\}$, where
$u_i$ is adjacent to $u_{i+1}$ for $1\le i<t$. See Figure~\ref{fig:B_G} for an
example. For each $i\in[t]$, let $d_i$ be equal to $d_{B_G}(u_i)$, and let
the neighborhood of $u_i$ in $B_G$ be
\[
    N_{B_G}(u_i)=\{v_{i,1},\dots,v_{i,d_i}\},
\]
so that vertices $v_{i,j}$ and $v_{i,j+1}$ are adjacent in $G$, for $1\le j<
d_i$. The vertex $u_i$ is the center of a maximal fan $F_i$ in $G$. Note that
$F_i$ is unique having $u_i$ as its center since $G$ is a $2$-path. We denote
the vertices of the rim path of $F_i$ by
\[
    v_{i,0},v_{i,1},\dots,v_{i,d_i},v_{i,d_i+1},
\]
so that $u_iv_{i,0}$ and $u_iv_{i,d_i+1}$ are the two edges of $G$ incident to
$u_i$ that do not belong to $B_G$, and $v_{i,0}v_{i,1}\cdots
v_{i,d_i}v_{i,d_i+1}$ is the rim path of $F_i$.  We call $v_{i,0}$ and
$v_{i,d_i+1}$ the \textit{boundary vertices} of $F_i$. 

\begin{figure}[t]
    \centering
    \includegraphics[width=0.8\textwidth]{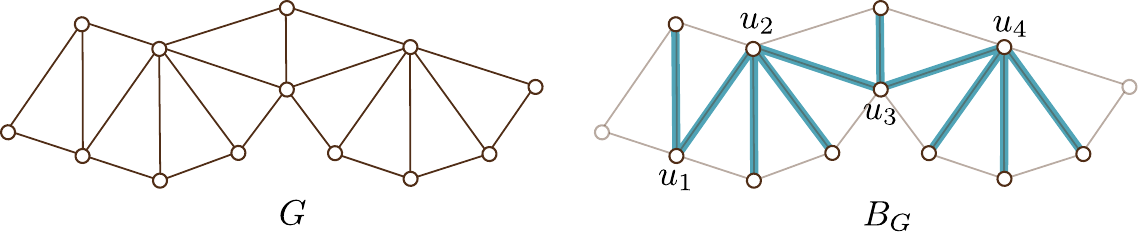}
    \caption{A $2$-path $G$ and its base graph $B_G$, which is a caterpillar
        with spine $(u_1,u_2,u_3,u_4)$. }
    \label{fig:B_G} 
\end{figure}

Consecutive fans (namely, those whose centers are consecutive vertices of
$\spn$) overlap in the following way: for every $i\in\{1,\dots,t-1\}$,
\[
v_{i,d_i}=u_{i+1},\qquad
v_{i,d_i-1}=v_{i+1,0},\qquad
v_{i,d_i+1}=v_{i+1,2},\qquad
v_{i+1,1}=u_i.
\]
Additionally, observe that $v_{1,0}$ is the simplicial vertex adjacent to $u_1$,
whereas $v_{t,d_t+1}$ is the other simplicial vertex of $G$ adjacent to $u_t$.

The following result will be used to reduce the study of induced $P_4$'s to
paths on three vertices contained in a single fan.
\begin{lemma}
\label{lem:P4-has-P3-in-fan}
    Let $G$ be a $2$-path, and let $F_1,\dots,F_t$ be the fans associated with
    the vertices of the spine of $B_G$ as above. If $P$ is an induced path on
    four vertices in $G$, then there exists $i\in[t]$ such that
    \[
        |V(P)\cap V(F_i)|\ge 3.
    \]
\end{lemma}
\begin{proof}
    It is enough to prove that every induced path on three vertices in $G$ is
    contained in some fan $F_i$. Indeed, if $P=(w,x,y,z)$ is an induced path on
    four vertices, then $(w,x,y)$ is an induced path on three vertices; hence
    its three vertices are contained in some fan $F_i$.

    Let $(x,y,z)$ be an induced path on three vertices. Suppose first that $y\in
    V(\spn)$, say $y=u_i$. Then all neighbors of $u_i$ in $G$ lie on the rim
    path of the fan $F_i$, namely among
    \[
        v_{i,0},v_{i,1},\dots,v_{i,d_i},v_{i,d_i+1}.
    \]
    Thus $x,z\in V(F_i)$, and so $x,y,z\in V(F_i)$.

    Suppose now that $y\notin V(\spn)$. Then $y$ lies on the rim path of at
    least one fan $F_i$. If $y$ is not a boundary vertex of such a fan, then all
    neighbors of $y$ in $G$ are contained in this same fan. Therefore $x,y,z\in
    V(F_i)$.

    It remains to consider the case in which $y$ is a boundary vertex. Then $y$
    belongs to two consecutive fans. By the overlap relations between
    consecutive fans, every induced path on three vertices with middle vertex
    $y$ is contained in one of these two fans. Since $(x,y,z)$ is induced and
    has middle vertex $y$, it follows that $x,y,z$ are contained in some fan
    $F_i$.

    Therefore every induced path on three vertices in $G$ is contained in some
    fan $F_i$, and the result follows.
\end{proof}


\subsection{\texorpdfstring{Minimal obstructions to star $4$-colorings of $2$-paths}{Minimal obstructions to star 4-colorings of 2-paths}}

In this subsection, we focus on the case $k=4$ and identify the minimal
$2$-paths that are not star $4$-colorable. 

\begin{figure}
  \centering
  \includegraphics[width=0.65\textwidth]{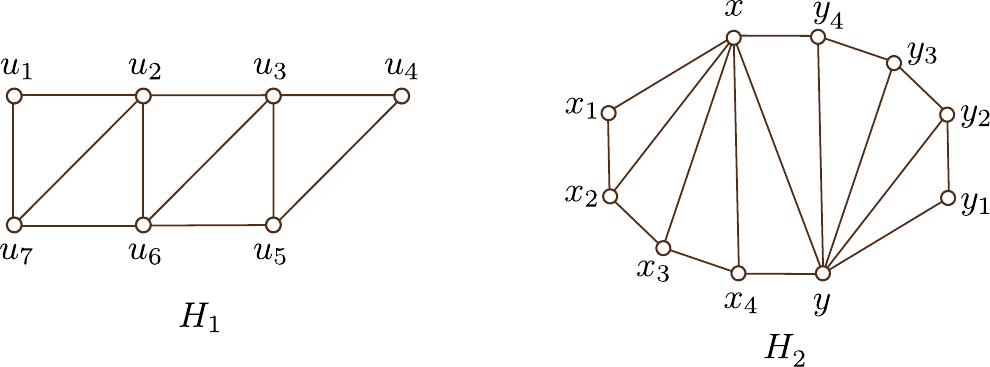}
  \caption{The two minimal $2$-paths that are not star $4$-colorable.}
  \label{fig:2-paths_k=4}
\end{figure}

\begin{proposition}
\label{pro:graphs_H1H2}
    The graphs $H_1$ and $H_2$ depicted in Figure~\ref{fig:2-paths_k=4} are
    minimal obstructions to star $4$-colorability within the class of
    $2$-paths. 
\end{proposition}
\begin{proof}
    We first show that $H_1$ is not star $4$-colorable. Suppose, for a
    contradiction, that $c$ is a star $4$-coloring of $H_1$. Note that the
    vertices $u_2,u_3,u_6$ induce a triangle, so they must receive pairwise
    distinct colors. We may assume $c(u_2)=\red$, $c(u_3)=\blue$, and
    $c(u_6)=\green$. Since $u_4$ and $u_5$ are adjacent to $u_3$, we must have
    $c(u_4) \in \{ \green,\purple \}$ and $c(u_5)\in \{ \red,\purple \}$.
    Moreover, it is not possible that $c(u_4)=\green$ and $c(u_5)=\red$, since
    then $(u_4,u_5,u_6,u_2)$ would be a bicolored $P_4$, a contradiction. 
    \begin{itemize}
        \item Suppose $c(u_4)=\purple$, so $c(u_5)=\red$. If $c(u_1)=\blue$
            (resp. $c(u_1)=\green$) then  $(u_1,u_2,u_3,u_5)$ (resp.
            $(u_1,u_2,u_6,u_5)$) is a bicolored $P_4$, a contradiction. Thus,
            $c(u_1)=\purple$, but then $u_7$ must be assigned $\blue$.
            Therefore, $(u_7,u_2,u_3,u_5)$ is a bicolored $P_4$, again a
            contradiction. 

        \item Suppose $c(u_4)=\green$, so $c(u_5)=\purple$. The vertex $u_7$ is
            adjacent to both $u_2$ and $u_6$, hence $c(u_7)\notin
            \{\red,\green\}$. If $c(u_7)=\blue$, then $(u_4,u_3,u_6,u_7)$ is a
            bicolored $P_4$, whereas if $c(u_7)=\purple$, then
            $(u_4,u_5,u_6,u_7)$ is a bicolored $P_4$; both cases contradict that
            $c$ is a star coloring. Thus, $u_7$ would require a fifth color,
            contradicting that $c$ uses only four colors.  
    \end{itemize}
    Therefore, $H_1$ is not star $4$-colorable. 

    We now show that $H_2$ is not star $4$-colorable. Suppose, for a
    contradiction, that $c$ is a star $4$-coloring of $H_2$. Without loss of
    generality, suppose $c(x)=\red$ and $c(y)=\blue$. When coloring the induced
    path $(x_1,x_2,x_3,x_4)$, we must use the color $\blue$ on at least one of
    its vertices; otherwise $(x_1,x_2,x_3,x_4)$ would have to alternate between
    the two remaining colors $\{\green,\purple\}$, yielding a bicolored $P_4$, a
    contradiction. Moreover, the blue vertex among $\{x_1,x_2,x_3,x_4\}$ cannot
    be $x_4$, since $x_4$ is adjacent to $y$ and $c(y)=\blue$. Thus,
    $c(x_i)=\blue$ for some $i\in\{1,2,3\}$. By symmetry, there exists
    $j\in\{1,2,3\}$ such that $c(y_j)=\red$. But then $(x_i, x, y, y_j)$ is a
    bicolored $P_4$, a contradiction. Thus, $H_2$ is not star $4$-colorable. 

    Finally, the minimality of graphs $H_1$ and $H_2$ is shown in
    Figure~\ref{fig:2-paths_4colors}, where (up to isomorphism) every graph
    obtained from $H_1$ or $H_2$ by deleting one vertex admits a star
    $4$-coloring.  
\end{proof}

Notice that the associated degree sequences of $H_1$ and $H_2$ are $\sigma(H_1)
= (2,2,2)$ and $\sigma(H_2) = (4,4)$.   We will use this information in later
sections.

\begin{figure}[t]
    \centering
    \includegraphics[width=1\textwidth]{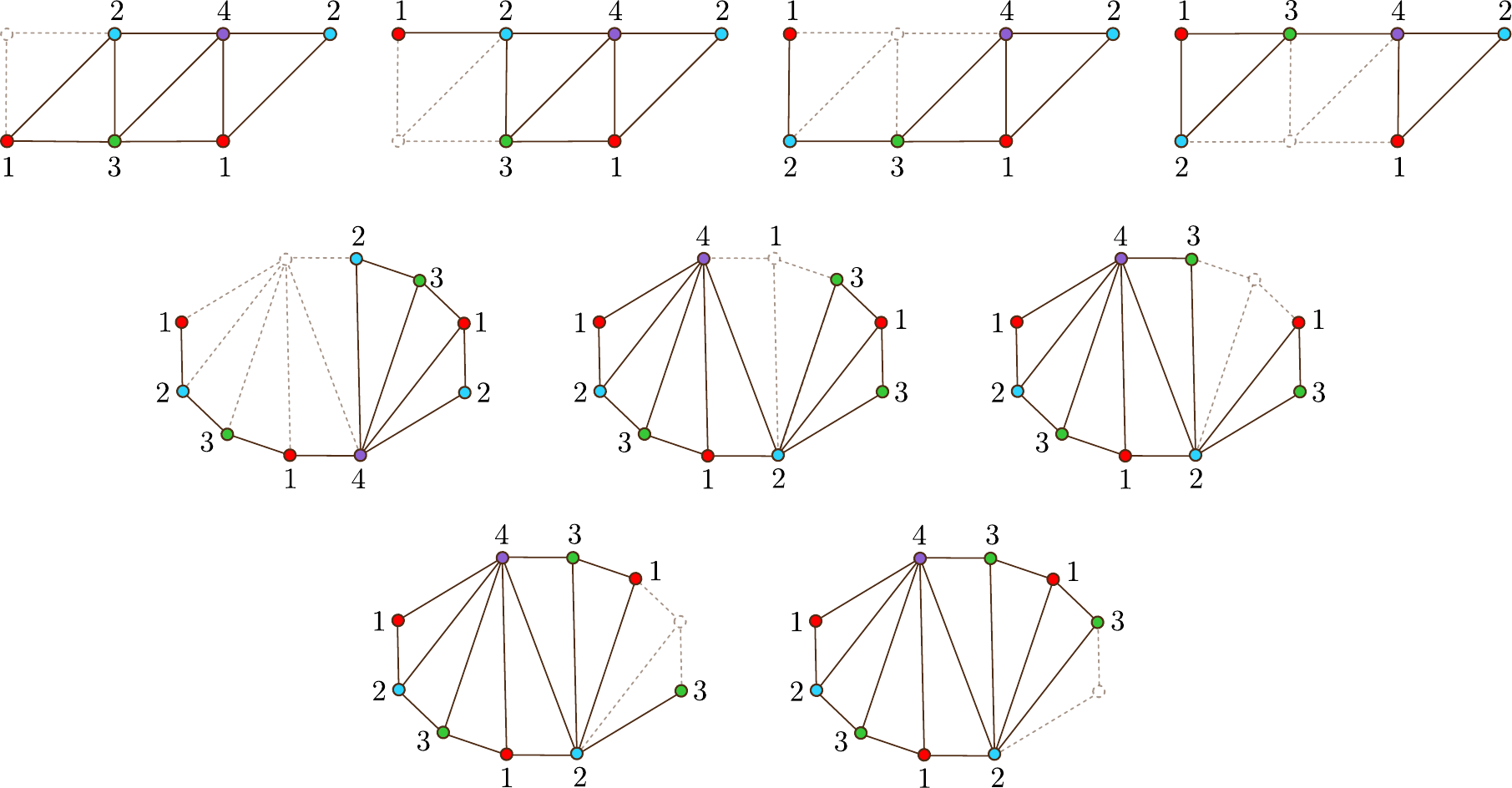}
    \caption{star $4$-colorings of all induced subgraphs obtained, up to
        isomorphism, by deleting one vertex from each of the two $2$-paths $H_1$
        and $H_2$. These colorings show that every proper induced subgraph of
        $H_1$ and $H_2$ is star $4$-colorable, and hence $H_1$ and $H_2$ are
        minimal obstructions to star $4$-colorability in the class of
        $2$-paths.}
    \label{fig:2-paths_4colors}
\end{figure}

\subsection{\texorpdfstring{Star $4$-coloring of an $\{H_1,H_2\}$-free $2$-path}{Star 4-coloring of an {H1,H2}-free 2-path}}
\label{subsec:2paths-4star-construction}

We define a coloring $c \colon V(G)\to\{\red,\blue,\green,\purple\}$ by
processing the fans $F_1,F_2,\dots,F_t$ in increasing ordering of the indices.
Recall that $d_i=d_{B_G}(u_i)\ge 2$ for each $i\in [t]$ since $u_i \in
V(\spn)$, so it is a non-leaf of $B_G$. Moreover, $d_i>2$ whenever $1<i<t$,
because $G$ is $H_1$-free, and  if there is $j\in [t]$ such that $d_j
\ge 4$, then $d_{j-1},d_{j+1}\le 3$, because $G$ is $H_2$-free.

\smallskip
\noindent\textit{Step 1: coloring the fan $F_1$.} Set $c(u_1)=\red$.
\begin{itemize}
    \item If $d_1=3$, then the fan $F_1$ has rim path
        $(v_{1,0},v_{1,1},v_{1,2},v_{1,3},v_{1,4})$. We assign
        \[
        c(v_{1,0})=c(v_{1,2})=c(v_{1,4})=\blue,\qquad
        c(v_{1,1})=\green,\qquad
        c(v_{1,3})=\purple.
        \]
    \item If $d_1\neq 3$, we color the rim vertices
        $v_{1,0},v_{1,1},\dots,v_{1,d_1}$ by repeating the pattern
        \[
        \blue,\green,\purple,\blue,\green,\purple,\dots
        \]
        in this order from $v_{1,0}$ up to $v_{1,d_1}$. We then color the
        boundary vertex $v_{1,d_1+1}$ as follows:
        \begin{itemize}
            \item If $t\ge 2$ and $d_2=3$, then set
                \[
                c(v_{1,d_1+1})=c(v_{1,d_1-1}).
                \]
            \item Otherwise, set $c(v_{1,d_1+1})$ to be the unique color not
                used in the triangle induced by $\{u_1,v_{1,d_1},v_{1,d_1-1}\}$.
    \end{itemize}
\end{itemize}

\smallskip
\noindent\textit{Step $i+1$: coloring the fan $F_{i+1}$ (for $i\ge 1$).} Assume
that we have already colored the vertices of the fans $F_1,\dots,F_{i}$. By the
overlap relations above, the vertices
\[
u_{i+1}=v_{i,d_i},\quad v_{i+1,0}=v_{i,d_i-1},\quad v_{i+1,1}=u_i,\quad
v_{i+1,2}=v_{i,d_i+1}
\]
are already colored. It remains to color $v_{i+1,3},\dots,v_{i+1,d_{i+1}+1}$.

\begin{itemize}
    \item If $d_{i+1}=3$, then we set $c(v_{i+1,3})$ to be the unique color that
        does not appear in the triangle induced by
        $\{u_{i+1},v_{i+1,1},v_{i+1,2}\}$, and we set
        \[
        c(v_{i+1,4})=c(v_{i+1,2}).
        \]
    \item If $d_{i+1}\neq 3$, then for each $j\in\{3,4,\dots,d_{i+1}\}$, in
        increasing order, we define $c(v_{i+1,j})$ to be the unique color not
        used in the triangle induced by $\{u_{i+1},v_{i+1,j-1},v_{i+1,j-2}\}$.
        Finally, we color the boundary vertex $v_{i+1,d_{i+1}+1}$ by the same
        rule, unless $i+2\le t$ and $d_{i+2}=3$, in which case we set
        \[
        c(v_{i+1,d_{i+1}+1})=c(v_{i+1,d_{i+1}-1}).
        \]
\end{itemize}
The above procedure defines a coloring $c$ of $G$ with four colors. See Figure~\ref{fig:coloring_k=4} for an example of the coloring $c$ applied to
a $2$-path $G$ that is $\{H_1,H_2\}$-free.

\begin{figure}[t]
    \centering
    \includegraphics[width=0.37\textwidth]{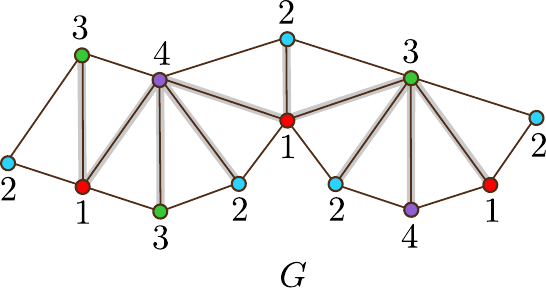}
    \caption{A star $4$-coloring of $G$ given by the coloring $c$ defined in
        Subsection~\ref{subsec:2paths-4star-construction}. The caterpillar $B_G$
        is highlighted in gray. Vertex labels encode the colors, with $1=\red$,
        $2=\blue$, $3=\green$, and $4=\purple$.}
    \label{fig:coloring_k=4}
\end{figure}

We first show that $c$ is proper, and describe how the bicolored paths on
three vertices are distributed through $G$.

\begin{lemma}
\label{lemma:4-coloring}
    The coloring $c:V(G)\to \{\red,\blue,\green,\purple\}$ is proper. Moreover,
    for any induced path $Q$ on three vertices contained in some fan $F_i$, if
    $Q$ is bicolored, then one of the following holds:
    \[
        Q=(v_{i,0},v_{i,1},v_{i,2}),\qquad
        Q=(v_{i,d_i-1},v_{i,d_i},v_{i,d_i+1}),
    \]
    or
    \[
        Q=(v_{i,j},u_i,v_{i,\ell})
        \quad\text{for some }0\le j<\ell\le d_i+1
        \text{ with }\ell\ge j+2.
    \]
\end{lemma}
\begin{proof}
    Let $F_1,\dots,F_t$ be the maximal fans of $G$. We first show that $c$ is
    proper. It is enough to prove that the restriction of $c$ to each fan $F_i$
    is proper. For $i=1$ this follows directly from the definition of $c$. Fix
    $i>1$ and suppose that the restriction of $c$ to each $F_j$ with $j<i$ is
    proper.

    By the overlap relations, the vertices
    \[
        u_i=v_{i-1,d_{i-1}},\qquad
        v_{i,0}=v_{i-1,d_{i-1}-1},\qquad
        v_{i,1}=u_{i-1},\qquad
        v_{i,2}=v_{i-1,d_{i-1}+1}
    \]
    are already colored. Since the triangle induced by $\{u_i,v_{i,1},v_{i,2}\}$
    belongs to $F_{i-1}$, by the induction hypothesis its vertices receive three
    distinct colors. Thus, after possibly permuting the colors, we may assume
    that
    \[
        c(u_i)=\red,\qquad c(v_{i,1})=\blue,\qquad c(v_{i,2})=\green.
    \]

    If $d_i=3$, then by definition
    \[
        c(v_{i,3})=\purple
        \qquad\text{and}\qquad
        c(v_{i,4})=c(v_{i,2})=\green.
    \]
    Since the only neighbors of $v_{i,4}$ in $F_i$ are $u_i$ and $v_{i,3}$, the
    fan $F_i$ is properly colored.

    If $d_i=2$, then, since $G$ is $H_1$-free, we must have $i=t$. By the
    coloring rule, $v_{i,3}$ receives the unique color not used in the triangle
    induced by $\{u_i,v_{i,1},v_{i,2}\}$, namely $\purple$. Hence $F_i$ is
    properly colored.

    Finally, suppose that $d_i\ge 4$. By the coloring rule, the vertices
    $v_{i,1},v_{i,2},\dots,v_{i,d_i}$ are colored periodically with three
    colors:
    \[
        \blue,\green,\purple,\blue,\green,\purple,\dots .
    \]
    Thus the rim edges among $v_{i,1},\dots,v_{i,d_i}$ are properly colored, and
    no vertex in this part receives color $\red$, so all edges incident with the
    center $u_i$ are properly colored. It remains to consider the boundary
    vertex $v_{i,d_i+1}$. By definition, this vertex receives either
    $c(v_{i,d_i-1})$ or the unique color not used in the triangle induced by
    $\{u_i,v_{i,d_i},v_{i,d_i-1}\}$; in the latter case this color is
    $c(v_{i,d_i-2})$. In both cases, its color is distinct from $c(u_i)$ and
    from $c(v_{i,d_i})$, which are precisely the colors of its neighbors in
    $F_i$. Hence $F_i$ is properly colored.

    Therefore, the restriction of $c$ to each fan $F_i$ is proper. Since every
    edge of $G$ belongs to some fan, $c$ is a proper coloring of $G$.

    \medskip
    We now prove the second part of the statement. Let $Q$ be an induced path on
    three vertices contained in some fan $F_i$, and suppose that $Q$ is
    bicolored.

    First assume that $Q$ is contained in the rim path of $F_i$. Then
    \[
        Q=(v_{i,j},v_{i,j+1},v_{i,j+2})
    \]
    for some $0\le j\le d_i-1$. If $j=0$, then $Q=(v_{i,0},v_{i,1},v_{i,2})$,
    and if $j=d_i-1$, then $Q=(v_{i,d_i-1},v_{i,d_i},v_{i,d_i+1})$; these are
    precisely the two boundary paths listed in the statement. Thus we may assume
    that $1\le j\le d_i-2$. In this case $Q$ is contained in the path
    \[
        (v_{i,1},v_{i,2},\dots ,v_{i,d_i}).
    \]
    By the definition of $c$, the colors on this path repeat periodically with
    period three: for every $q\in\{3,\dots,d_i\}$ we have
    \[
        c(v_{i,q})=c(v_{i,q-3}).
    \]
    Hence any three consecutive vertices on this path receive three distinct
    colors, contradicting that $Q$ is bicolored. Therefore, if a bicolored
    induced path on three vertices is contained in the rim path, then it must be
    one of the two boundary paths listed in the statement.

    It remains to consider the case in which $Q$ contains the center $u_i$.
    Since $c$ is proper, the center cannot have the same color as either of its
    neighbors in $Q$. Hence, if $Q$ is bicolored, then the two endpoints of $Q$
    must have the same color. Thus $u_i$ is the middle vertex of $Q$, and we can
    write
    \[
        Q=(v_{i,j},u_i,v_{i,\ell})
    \]
    for some $0\le j<\ell\le d_i+1$. Since $Q$ is induced, the vertices
    $v_{i,j}$ and $v_{i,\ell}$ are not adjacent in the rim path of $F_i$; hence
    $\ell\ge j+2$. This is precisely the third possibility in the statement.
\end{proof}

We are now ready to prove the characterization of star $4$-colorable
$2$-paths. 

\begin{theorem}
\label{thm:2paths-k4}
    Let $G$ be a $2$-path. Then $G$ is star $4$-colorable if and only if $G$ is
    $\{H_1,H_2\}$-free, where $H_1$ and $H_2$ are the graphs depicted in
    Figure~\ref{fig:2-paths_k=4}. In particular, $H_1$ and $H_2$ are exactly the
    minimal obstructions to star $4$-colorability within the class of
    $2$-paths.
\end{theorem}
\begin{proof}
    By Proposition~\ref{pro:graphs_H1H2}, neither $H_1$ nor $H_2$ is star
    $4$-colorable. Hence, if $G$ is star $4$-colorable, then $G$ is
    $\{H_1,H_2\}$-free. 

    Conversely, suppose that $G$ is $\{H_1,H_2\}$-free. Let
    $c:V(G)\to\{\red,\blue,\green,\purple\}$ be the coloring defined above. We
    show that $c$ is a star $4$-coloring of $G$. By
    Lemma~\ref{lemma:4-coloring}, the coloring $c$ is proper, so it remains to
    show that no induced path on four vertices is bicolored. Suppose, for a
    contradiction, that $P=(w,x,y,z)$ is a bicolored induced path on four
    vertices, so the colors on $P$ alternate; in particular, both $(w,x,y)$ and
    $(x,y,z)$ are bicolored induced paths on three vertices.

    First note that $P$ cannot be contained in a single fan. Indeed, suppose
    that $V(P)\subseteq V(F_i)$ for some $i\in[t]$. Since $P$ is induced, it
    must be contained in the rim path of $F_i$, say
    $P=(v_{i,j},v_{i,j+1},v_{i,j+2},v_{i,j+3})$, with $0\le j\le d_i-2$. If
    $d_i>2$, then either $j>0$ or $j+3<d_i+1$. In the first case,
    $(v_{i,j},v_{i,j+1},v_{i,j+2})$ is not one of the two boundary paths listed
    in Lemma~\ref{lemma:4-coloring}; in the second case,
    $(v_{i,j+1},v_{i,j+2},v_{i,j+3})$ is not one of them. Thus one of these two
    subpaths is not bicolored, contradicting that $P$ is bicolored. If $d_i=2$,
    then $i=1$ or $i=t$ because $G$ is $H_1$-free; in this case the rim path of
    $F_i$ has exactly four vertices and, by the definition of the coloring,
    receives at least three colors, again a contradiction.

    Therefore $P$ is not contained in a single fan. By
    Lemma~\ref{lem:P4-has-P3-in-fan}, and after reversing the order of $P$ if
    necessary, we may assume that $(w,x,y)$ is contained in some fan $F_i$ and
    that $z\in V(F_{i+1})\setminus V(F_i)$. By Lemma~\ref{lemma:4-coloring}, the
    bicolored path $(w,x,y)$ is one of the following types:
    \[
        (v_{i,0},v_{i,1},v_{i,2}),\qquad
        (v_{i,d_i-1},v_{i,d_i},v_{i,d_i+1}),
    \]
    or
    \[
        (v_{i,j},u_i,v_{i,\ell})
        \quad\text{for some }0\le j<\ell\le d_i+1
        \text{ with }\ell\ge j+2.
    \]
    We show that none of these possibilities can be extended to the vertex $z\in
    V(F_{i+1})\setminus V(F_i)$.

    \medskip
    \noindent\textbf{Case 1.} Suppose $(w,x,y)=(v_{i,0},v_{i,1},v_{i,2})$. Since
    $z$ is adjacent to $y=v_{i,2}$ and lies outside $F_i$, the vertex $v_{i,2}$
    must be one of the vertices through which $F_i$ meets $F_{i+1}$. This is
    possible only if $d_i=2$, in which case $v_{i,2}=v_{i,d_i}$. Since $G$ is
    $H_1$-free, this forces $i=1$. But then the vertices $w,x,y$ receive the
    colors $\blue,\green,\purple$, respectively, contradicting that the path
    $(w,x,y)$ is bicolored.

    \medskip
    \noindent\textbf{Case 2.} Suppose
    $(w,x,y)=(v_{i,d_i-1},v_{i,d_i},v_{i,d_i+1})$. Then $x=v_{i,d_i}=u_{i+1}$
    and $y=v_{i,d_i+1}=v_{i+1,2}$. Since $z\in V(F_{i+1})\setminus V(F_i)$ is
    adjacent to $y$, the only possible choice is $z=v_{i+1,3}$. But $z$ is
    adjacent to $x=u_{i+1}$, contradicting that $P=(w,x,y,z)$ is induced.

    \medskip
    \noindent\textbf{Case 3.} Suppose $(w,x,y)=(v_{i,j},u_i,v_{i,\ell})$, with
    $\ell\ge j+2$. Thus $x=u_i$ and $y=v_{i,\ell}$. Since $z$ lies outside
    $F_i$, is adjacent to $y$, and is not adjacent to $x=u_i$, we must have
    $y\in\{v_{i,d_i},v_{i,d_i+1}\}$.

    Suppose first that $y=v_{i,d_i+1}$. By the overlap relations,
    $x=u_i=v_{i+1,1}$ and $y=v_{i,d_i+1}=v_{i+1,2}$. Since $P$ is induced, the
    only possible choice is $z=v_{i+1,3}$. If $d_{i+1}=2$, then by the
    definition of $c$ we have $c(z)\ne c(x)$; if $d_{i+1}\ge 3$, then
    Lemma~\ref{lemma:4-coloring} implies that the path
    $(x,y,z)=(v_{i+1,1},v_{i+1,2},v_{i+1,3})$ is not bicolored. In both cases we
    get a contradiction.

    We may therefore assume that $y=v_{i,d_i}=u_{i+1}$. Then $x=u_i=v_{i+1,1}$,
    and $(x,y,z)$ is a bicolored induced path contained in $F_{i+1}$. If
    $d_{i+1}=2$, then, because $P$ is induced, we must have
    $z=v_{i+1,d_{i+1}+1}$, and the definition of $c$ gives $c(z)\ne c(x)$, a
    contradiction. If $d_{i+1}=3$, then $z=v_{i+1,r}$ for some $r\in\{3,4\}$,
    and the definition of $c$ again gives $c(z)\ne c(x)$, a contradiction.

    Thus $d_{i+1}\ge 4$. Since $G$ is $H_2$-free, it follows that $d_i\le 3$. If
    $d_i=2$, then $i=1$ because $G$ is $H_1$-free. Moreover, as $P$ is induced
    and $y=v_{i,2}$, we have $w=v_{i,0}$; hence $c(w)=\blue\ne\purple=c(y)$,
    contradicting that $(w,x,y)$ is bicolored.

    It remains to consider the case $d_i=3$. Then $y=v_{i,3}$ and, since
    $(w,x,y)=(v_{i,j},u_i,v_{i,3})$ is induced, we have $j\in\{0,1\}$. If
    $w=v_{i,1}$, then $c(w)\ne c(y)$ because $c(y)$ is the unique color not used
    in the triangle induced by $\{u_i,v_{i,1},v_{i,2}\}$. Hence we may assume
    that $w=v_{i,0}$. If $i=1$, then the definition of the coloring of $F_1$
    gives $c(v_{i,0})\ne c(v_{i,3})$, again a contradiction. Thus $i>1$. We
    claim that $c(v_{i,0})=c(v_{i,2})$. Indeed, by the overlap relations,
    $v_{i,0}=v_{i-1,d_{i-1}-1}$ and $v_{i,2}=v_{i-1,d_{i-1}+1}$. If $d_{i-1}=3$,
    then the coloring rule gives $c(v_{i-1,4})=c(v_{i-1,2})$, which is precisely
    $c(v_{i,2})=c(v_{i,0})$. If $d_{i-1}\ne 3$, then, since $d_i=3$, the special
    boundary rule applied to $F_{i-1}$ gives
    $c(v_{i-1,d_{i-1}+1})=c(v_{i-1,d_{i-1}-1})$, and again
    $c(v_{i,2})=c(v_{i,0})$. Now $c(y)=c(v_{i,3})$ is the unique color not
    appearing in the triangle induced by $\{u_i,v_{i,1},v_{i,2}\}$. Since
    $c(v_{i,0})=c(v_{i,2})$, we obtain $c(y)\ne c(v_{i,0})=c(w)$, contradicting
    that $(w,x,y)$ is bicolored.

    In all cases we obtain a contradiction. Therefore no induced path on four
    vertices is bicolored, and so $c$ is a star $4$-coloring of $G$. 

    Therefore, $G$ is star $4$-colorable if and only if it is
    $\{H_1,H_2\}$-free. The final statement follows from
    Proposition~\ref{pro:graphs_H1H2}, which shows that $H_1$ and $H_2$ are
    minimal with this property within the class of $2$-paths. 
\end{proof}

\begin{figure}[t]
    \centering
    \includegraphics[width=0.8\textwidth]{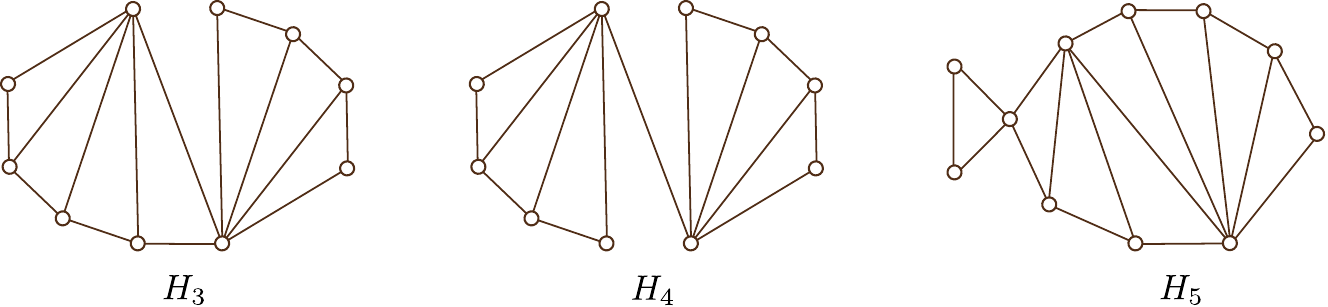}
    \caption{Three minimal obstructions to star $4$-colorability within the class
of partial $2$-paths.}
    \label{fig:H's_1-5}
\end{figure}

We have proved that every $\{H_1,H_2\}$-free $2$-path is star $4$-colorable.
However, the class of $2$-paths is not hereditary: an induced subgraph of a
$2$-path need not be a $2$-path. For this reason, it is natural to consider the
hereditary closure of this class, namely the class of partial $2$-paths. It is
not hard to check that the graphs $H_3,H_4,H_5$ depicted in
Figure~\ref{fig:H's_1-5} are minimal obstructions to star $4$-colorability
within the class of partial $2$-paths.   It is natural to ask whether these
graphs completely characterize star $4$-colorable partial $2$-paths.

\begin{question}\label{ques:partial-2paths-4star}
    Are $H_1,H_2,H_3,H_4,H_5$ the only minimal obstructions to star
    $4$-colorability within the class of partial $2$-paths?
\end{question}


\subsection{\texorpdfstring{$2$-paths are star $5$-colorable}{2-paths are star 5-colorable}}

\label{sec:5-coloring}

Let $G$ be a $2$-path and let $B_G$ be its base caterpillar. Denote by
$F_1,\dots,F_t$ the maximal fans of $G$. If $t=1$, then $G$ is a single fan and
we already know that $G$ is star $4$-colorable. In this subsection, we
nevertheless present a unified construction that yields a star $5$-coloring for
any $2$-path $G$.

\medskip
\noindent\textit{Step 1: coloring the fan $F_1$.} Set $c(u_1)=\red$.
\begin{itemize}
    \item If $d_1=2$, define
    \[
        c(v_{1,0})=\blue,\quad
        c(v_{1,1})=\green,\quad
        c(v_{1,2})=\purple,\quad
        c(v_{1,3})=\orange.
    \]

    \item If $d_1\ge 3$, color the rim vertices
    $v_{1,0},v_{1,1},\dots,v_{1,d_1-1}$ by repeating the pattern
    \[
        \blue,\ \green,\ \purple,\ \blue,\ \green,\ \purple,\ \dots
    \]
    in this order from $v_{1,0}$ up to $v_{1,d_1-1}$. Additionally, we color the
    vertices $v_{1,d_1}$ and $v_{1,d_1+1}$ as follows:
    \begin{itemize}
        \item If $t\ge 2$ and $d_2=2$, set
        \[
            c(v_{1,d_1}) = c(v_{1,d_1-3}),\quad
            c(v_{1,d_1+1})= \text{the unique color not used on }
            \{u_1,v_{1,0},\dots,v_{1,d_1}\}.
        \]

        \item If $t=1$ or $d_2\ge 3$, set
        \[
            c(v_{1,d_1})= \text{the unique color not used on }
            \{u_1,v_{1,0},\dots,v_{1,d_1-1}\},\quad
            c(v_{1,d_1+1}) = c(v_{1,d_1-2}).
        \]
    \end{itemize}
\end{itemize}

\medskip
\noindent\textit{Step $i$: coloring the fan $F_i$ (for $i\ge 2$).} Assume that
we have already colored the vertices of the fans $F_1,\dots,F_{i-1}$. By the
overlap relations, the vertices
\[
u_i=v_{i-1,d_{i-1}},\quad
v_{i,0}=v_{i-1,d_{i-1}-1},\quad
v_{i,1}=u_{i-1},\quad
v_{i,2}=v_{i-1,d_{i-1}+1}
\]
are already colored. It remains to color the vertices
$v_{i,3},\dots,v_{i,d_i+1}$, and we proceed according to the value of $d_i$.

\begin{itemize}
    \item If $d_i=2$, only the boundary vertex $v_{i,3}$ remains uncolored. We
    set
    \[
        c(v_{i,3}) = \text{the unique color not used on }
        \{u_i,v_{i,0},v_{i,1},v_{i,2}\}.
    \]

    \item Suppose $d_i\ge 3$.
    \begin{itemize}
        \item If $t\ge i+1$ and $d_{i+1}=2$, we set
        \[
            c(v_{i,j}) = c(v_{i,j-3}) \qquad \text{for } 3\le j\le d_i,
        \]
        and
        \[
            c(v_{i,d_i+1}) = \text{the unique color not used on }
            \{u_i,v_{i,0},\dots,v_{i,d_i}\}.
        \]

        \item If $t=i$ or $d_{i+1}\ge 3$, we set
        \[
            c(v_{i,j}) = c(v_{i,j-3}) \qquad
            \text{for } 3\le j\le d_i-1 \text{ and for } j=d_i+1,
        \]
        and
        \[
            c(v_{i,d_i}) = \text{the unique color not used on }
            \{u_i,v_{i,0},\dots,v_{i,d_i-1}\}.
        \]
    \end{itemize}
\end{itemize}

\medskip
The above procedure defines a coloring $c$ of $G$ with the five colors
\[
    \{\red,\blue,\green,\purple,\orange\}.
\]
The fan $F_1$ uses at most these five colors, and in the recursive step each new
vertex either repeats a previously used color or is assigned a new color chosen
from this same set of five colors. An example is shown in Figure~\ref{fig:c_5}.

\begin{figure}[t]
    \centering
    \includegraphics[width=1\textwidth]{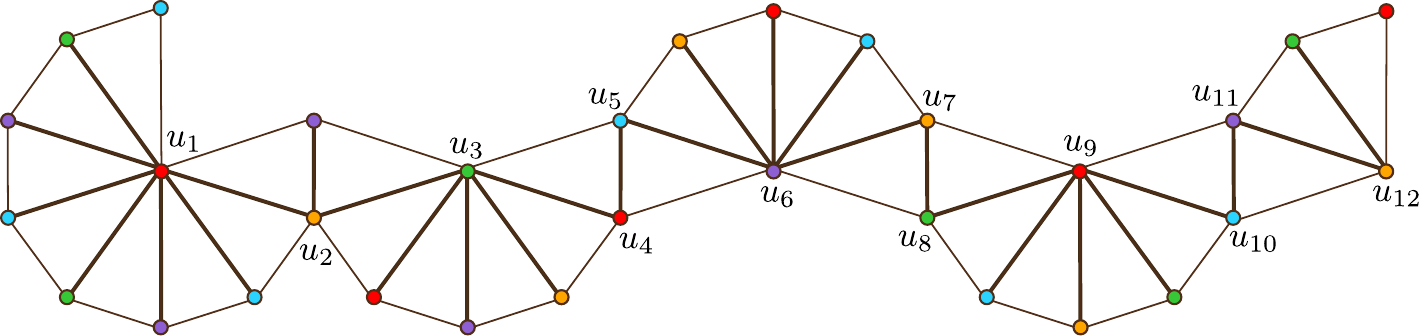}
    \caption{A $2$-path $G$ together with its base caterpillar $B_G$ (drawn with
    thicker edges). The labelled vertices $u_1,\dots,u_{12}$ form the spine of
    $B_G$ and are the centres of the maximal fans $F_1,\dots,F_{12}$ of $G$. The
    vertex colors illustrate the star $5$-coloring $c$ defined in
    Subsection~\ref{sec:5-coloring}. }
    \label{fig:c_5}
\end{figure}

We first show that $c$ is a proper coloring of $G$.

\begin{proposition}
\label{pro:c_proper}
    The coloring $c$ is a proper $5$-coloring of $G$. 
\end{proposition}
\begin{proof}
    As noted above, $c$ uses only five colors, so it remains to show that no
    edge of $G$ is monochromatic. It is sufficient to prove that the restriction
    of $c$ to each $F_i$ is proper. For $i=1$ this follows directly from the
    definition of $c$. Fix $i>1$ and suppose the statement holds for each $F_j$
    with $j<i$.
    
    If $d_i=2$, then $c(v_{i,3})$ is chosen as the unique color not used on
    $\{u_i,v_{i,0},v_{i,1},v_{i,2}\}$. In particular, $c(v_{i,3})$ differs from
    the colors on its neighbors in $F_i$ (which are $u_i$ and $v_{i,2}$), so no
    edge of $F_i$ incident with $v_{i,3}$ is monochromatic. All remaining edges
    of $F_i$ lie in $F_{i-1}$, and hence are properly colored by the induction
    hypothesis.
    
    Suppose now that $d_i\ge 3$. First note that $c(v_{i,0})\ne c(v_{i,2})$.
    Indeed, if $d_{i-1}=2$, this follows directly from the definition of
    $c(v_{i-1,3})$, where $v_{i-1,3}=v_{i,2}$; if $d_{i-1}\ge 3$, then
    \[
        c(v_{i-1,d_{i-1}+1})\ne c(v_{i-1,d_{i-1}-1}),
    \]
    where $v_{i-1,d_{i-1}+1}=v_{i,2}$ and $v_{i-1,d_{i-1}-1}=v_{i,0}$. Thus
    $c(u_i),c(v_{i,0}),c(v_{i,1}),c(v_{i,2})$ are four distinct colors.
    
    Moreover, for each $3\le j\le d_i+1$ either $c(v_{i,j})=c(v_{i,j-3})$ or
    $c(v_{i,j})$ is chosen as the unique color not used on the previously
    colored neighbors of $v_{i,j}$ listed in the construction. In either case,
    $v_{i,j}$ receives a color different from that of each of its neighbors in
    $F_i$. Hence, no edge of $F_i$ is monochromatic in this case either. 
    
    Therefore, the restriction of $c$ to $F_i$ is proper for every $i\in[t]$,
    and $c$ is a proper $5$-coloring of $G$.
\end{proof}

We now show some basic properties of the coloring $c$.

\begin{proposition}
\label{pro:c_basic}
    Let $G$ be a $2$-path with the coloring $c$ defined above. For each maximal
    fan $F_i$ of $G$, the following holds.
    \begin{enumerate}[label=$(\roman*)$]
        \item If $d_i=2$, then $V(F_i)$ uses all five colors, each color
            appearing exactly once.
        \item For each $j\in \{2,\dots,d_i+1\}$ we have $c(v_{i,j})\ne
            c(v_{i,j-2})$. Consequently, there is no bicolored path on three
            vertices contained in the rim path of $F_i$.
    \end{enumerate}
    In particular, if $Q$ is a bicolored induced path on three vertices in $G$,
    then $Q$ is of the form $(v_{i,j},u_i,v_{i,\ell})$, for some $i\in[t]$ and
    $0\le j<\ell\le d_i+1$, with $\ell\ge j+3$ and $d_i\ge 3$. 
\end{proposition}
\begin{proof}
    We proceed by induction on $i$. It is straightforward to check that
    properties $(i)$ and $(ii)$ hold for $F_1$. Fix $i>1$ and assume that the
    proposition holds for all $F_j$ with $j<i$.

    \smallskip
    \noindent\textit{Property $(i)$.} Suppose that $d_i=2$. Since
    $\{u_i,v_{i,0},v_{i,1}\}$ induces a triangle in $G$, they receive pairwise
    distinct colors (because $c$ is a proper coloring by
    Proposition~\ref{pro:c_proper}). Moreover, by property~$(ii)$ applied to
    $F_{i-1}$, we have $c(v_{i,2})\ne c(v_{i,0})$, and as
    $\{u_i,v_{i,1},v_{i,2}\}$ induces also a triangle, we obtain
    \[
        c(v_{i,2})\notin\{c(v_{i,1}),c(u_i)\}.
    \]
    Finally, by definition, $c(v_{i,3})$ is chosen as the unique color not used
    on $\{u_i,v_{i,0},v_{i,1},v_{i,2}\}$. Hence, all five colors appear on
    $V(F_i)=\{u_i,v_{i,0},v_{i,1},v_{i,2},v_{i,3}\}$, each exactly once, and
    $(i)$ holds.

    \smallskip
    \noindent\textit{Property $(ii)$.} The second part of $(ii)$ follows
    immediately from the first: if a subpath $(x,y,z)$ of the rim were
    bicolored, then $c(x)=c(z)$ and we would have $c(v_{i,j})=c(v_{i,j-2})$ for
    the appropriate index $j$. Thus, it suffices to show that $c(v_{i,j})\ne
    c(v_{i,j-2})$ for every $j\in\{2,\dots,d_i+1\}$. 

    If $d_i=2$, then $V(F_i)$ uses five distinct colors by $(i)$, so
    $c(v_{i,j})\ne c(v_{i,j-2})$ holds trivially for $j=2,3$. Hence, we may
    assume $d_i\ge 3$.

    By the induction hypothesis applied to $F_{i-1}$, we have that the vertices
    $v_{i,0}=v_{i-1,d_{i-1}-1}$ and $v_{i,2}=v_{i-1,d_{i-1}-1}$ receive distinct
    colors. For $3\le j\le d_i-1$ the definition of $c$ gives
    \[
        c(v_{i,j})=c(v_{i,j-3})\ne c(v_{i,j-2}),
    \]
    where the inequality follows from the fact that $v_{i,j-3}$ and $v_{i,j-2}$
    are adjacent and thus receive distinct colors because $c$ is proper by
    Proposition~\ref{pro:c_proper}. 

    It remains to consider $j\in\{d_i,d_i+1\}$. By the construction, for each of
    these indices either
    $c(v_{i,j})=c(v_{i,j-3})\ne c(v_{i,j-2})$,
    or $c(v_{i,j})$ is chosen as the unique color not used on the set
    $\{u_i,v_{i,0},\dots,v_{i,j-1}\}$, and in particular is different from
    $c(v_{i,j-2})$. Thus in all cases $c(v_{i,j})\ne c(v_{i,j-2})$ for every
    $j\in\{2,\dots,d_i+1\}$, and $(ii)$ follows.

    We now show the final assertion. By Lemma~\ref{lem:P4-has-P3-in-fan}, $Q$
    must be contained in some fan $F_i$. If $Q$ were contained in the rim path
    of $F_i$, then it would be a path on three consecutive rim vertices. This is
    impossible by $(i)$ and $(ii)$, since no path on three vertices contained in
    the rim path of $F_i$ is bicolored. Hence, $Q$ must contain the center
    $u_i$. Since $c$ is proper, $u_i$ must be the internal vertex of $Q$, and so
    the two endpoints of $Q$ must have the same color. Thus, we may write
\[
    Q=(v_{i,j},u_i,v_{i,\ell}) \quad \textrm{for some $0\le j<\ell\le d_i+1$.}
\]

By $(i)$ the five vertices of $F_i$ receive pairwise distinct colors, so $d_i\ge
3$. Finally, since $Q$ is induced, the endpoints $v_{i,j}$ and $v_{i,\ell}$ are
not adjacent on the rim path, and by $(ii)$, we conclude that  $\ell\ge j+3$, as
claimed. 
\end{proof}

We are now ready to prove the star $5$-colorability of $2$-paths.

\begin{theorem}
\label{thm:c-5-star}
Every $2$-path is star $5$-colorable.
\end{theorem}
\begin{proof}
    Let $G$ be a $2$-path, and let $c$ be the coloring of $G$ defined above. We
    show that $c$ is a star $5$-coloring of $G$. By
    Proposition~\ref{pro:c_proper}, the coloring $c$ is proper. We show that no
    path on four vertices is bicolored. Since $G$ is chordal and $c$ is proper,
    it is enough to consider induced paths on four vertices.

    Suppose, for a contradiction, that $P=(w,x,y,z)$ is a bicolored induced path
    on four vertices. Since $c$ is proper, the colors on $P$ alternate; in
    particular, both $(w,x,y)$ and $(x,y,z)$ are bicolored induced paths on
    three vertices.

    If all vertices of $P$ were contained in a single fan $F_i$, then, since $P$
    is induced, $P$ would have to be contained in the rim path of $F_i$. But
    this is impossible by Proposition~\ref{pro:c_basic}, because then some path
    on three consecutive rim vertices would be bicolored. Thus $P$ is not
    contained in a single fan. By  Lemma~\ref{lem:P4-has-P3-in-fan}, every
    induced path on three vertices is contained in some fan. Hence, after
    reversing the order of $P$ if necessary, we may assume that $(w,x,y)$ is
    contained in some fan $F_i$ and that $z\in V(F_{i+1})\setminus V(F_i)$. By
    Proposition~\ref{pro:c_basic}, the bicolored path $(w,x,y)$ must be of the
    form $(w,x,y)=(v_{i,j},u_i,v_{i,\ell})$, where $0\le j<\ell\le d_i+1$,
    $\ell\ge j+3$, and $d_i\ge 3$. In particular, $x=u_i$, $w=v_{i,j}$ and
    $y=v_{i,\ell}$, with $c(w)=c(y)$.

    Since $z\notin V(F_i)$, the vertex $y=v_{i,\ell}$ must be one of the
    vertices through which $F_i$ meets $F_{i+1}$. Thus
    $y\in\{v_{i,d_i},v_{i,d_i+1}\}$. Suppose first that $y=v_{i,d_i+1}$. By the
    overlap relations, we have $x=u_i=v_{i+1,1}$ and $y=v_{i,d_i+1}=v_{i+1,2}$.
    Since $z$ is adjacent to $y$, not adjacent to $x$, and lies in
    $F_{i+1}\setminus F_i$, the only possible choice is $z=v_{i+1,3}$. Hence
    $(x,y,z)$ is a path on three consecutive vertices of the rim path of
    $F_{i+1}$, contradicting Proposition~\ref{pro:c_basic}. Suppose now that
    $y=v_{i,d_i}=u_{i+1}$, and the path $(x,y,z)$ is a bicolored induced path on
    three vertices contained in $F_{i+1}$. By Proposition~\ref{pro:c_basic}, we
    must have $d_{i+1}\ge 3$. Hence, by the definition of the coloring in $F_i$,
    the color of $v_{i,d_i}$ was chosen as the unique color not used on
    $\{u_i,v_{i,0},\dots,v_{i,d_i-1}\}$. Note that $w=v_{i,j}$ belongs to this
    set (because $j\le \ell-3$), and so $c(y)=c(v_{i,d_i})\ne c(w)$,
    contradicting that $(w,x,y)$ is bicolored. This completes the proof. 
\end{proof}


\subsection{\texorpdfstring{The star chromatic number of $2$-paths is smaller than that of $2$-trees}{The star chromatic number of 2-paths is smaller than that of 2-trees}}

In Subsection~\ref{sec:5-coloring} we proved that every $2$-path is star
$5$-colorable. In this subsection we provide a $2$-tree whose star chromatic
number is $6$ (see Figure~\ref{fig:2-tree}). Consequently, the maximum star
chromatic number attained within the family of $2$-paths is strictly smaller
than the one attained within the family of $2$-trees. The existence of a
$2$-tree (on $48$ vertices) with star chromatic number $6$ was already exhibited
in~\cite{fertinJGT47}, and another one (with $41$ vertices) in
\cite{albertsonEJC11}; however, our example is noticeably smaller (on $21$
vertices) and minimal, i.e., any proper subgraph of $H$ admits a star
$5$-coloring. 

\begin{figure}
    \centering
    \includegraphics[width=0.5\linewidth]{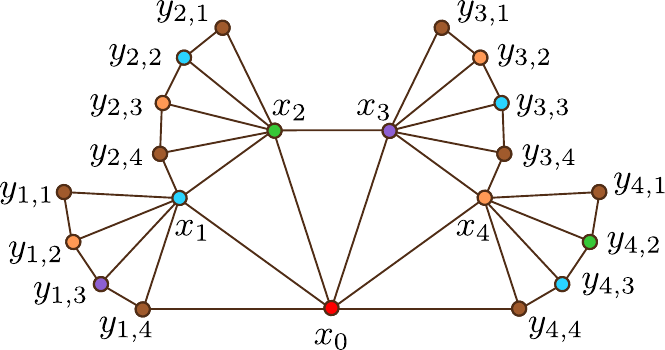}
    \caption{The $2$-tree $H$ from Proposition~\ref{pro:2-tree} together with a
    star $6$-coloring.}
    \label{fig:2-tree}
\end{figure}

\begin{proposition}
\label{pro:2-tree}
    The $2$-tree $H$ depicted in Figure~\ref{fig:2-tree} satisfies $\x(H)=6$.
    Moreover, $H$ is minimal, that is, any proper subgraph $H'$ of $H$ satisfies
    $\x(H')=5$.  
\end{proposition}
\begin{proof}
    Let us first establish $\x(H)$. For the upper bound, the coloring depicted
    in Figure~\ref{fig:2-tree} is a star $6$-coloring of $H$, thus implying that
    $\x(H)\le 6$.

    For the lower bound, we show that $H$ cannot be star colored with $5$
    colors. Suppose otherwise and let
    \[
        c:V(H)\to \{\red,\blue,\green,\purple,\orange\}
    \]
    be a star $5$-coloring. Without loss of generality, assume $c(x_0)=\red$.
    For $i \in \{1,2,3,4\}$ set
    \[
        Y_i:=\{y_{i,1},y_{i,2},y_{i,3},y_{i,4}\}.
    \]
    Clearly, for each $i\in\{1,2,3,4\}$ no vertex of $Y_i$ receives the color
    $c(x_i)$, since $x_i$ is adjacent to every vertex in $Y_i$.

    \medskip\noindent \textit{Case 1.} There exist indices $1\le i<j\le 4$ such
    that $c(x_i)=c(x_j)$.

    Among all such pairs choose $(i,j)$ with $j-i$ minimum. Then $j-i\ge 2$
    since $c$ is proper. Renaming colors if necessary, we may assume
    $c(x_i)=c(x_j)=\blue$. Moreover, by the minimality of $j-i$ we have
    $c(x_\ell)\ne\blue$ for every $\ell$ with $i<\ell<j$.

    Since $x_0$ is adjacent to both $x_i$ and $x_j$, we have $c(x_0)\ne
    c(x_i)=\blue$. As $c(x_0)=\red$, we obtain:

    \begin{equation}
    \label{eq:no-red}
        \text{no vertex in $Y_i\cup Y_j$ receives the color $\red$,}
    \end{equation}
    because otherwise such a vertex together with $x_0,x_i,x_j$ would induce a
    bicolored $P_4$ in $H$, contradicting that $c$ is a star coloring.

    We distinguish two subcases according to the value of $j-i$.

    \medskip\noindent \textit{Subcase 1.1.} $j-i=2$.

    Then $(i,j)=(1,3)$ or $(2,4)$; by symmetry of $H$ we may assume
    $(i,j)=(1,3)$. Renaming colors among $\{\blue,\green,\purple,\orange\}$ if
    necessary, we may assume $c(x_1)=c(x_3)=\blue$ and $c(x_2)=\green$.

    By~\eqref{eq:no-red}, no vertex of $Y_3$ receives color $\red$, and as
    observed above no vertex of $Y_3$ receives color $c(x_3)=\blue$. Hence the
    vertices of $Y_3$ can only use the three colors
    $\{\green,\purple,\orange\}$. Since $Y_3$ induces a $P_4$ and $c$ is a star
    coloring, this $P_4$ cannot be bicolored, so at least three colors must
    appear on $Y_3$. Therefore $Y_3$ uses all three colors
    $\{\green,\purple,\orange\}$, and in particular there exists $y_{3,q}\in
    Y_3$ with $c(y_{3,q})=\green$.

    Consider now the path $(x_1, x_2, x_3, y_{3,q})$.
    Its colors are
    \[
        c(x_1)=\blue,\quad c(x_2)=\green,\quad c(x_3)=\blue,\quad
        c(y_{3,q})=\green.
    \]
    Thus this path is a bicolored $P_4$, contradicting that $c$ is a star
    coloring. 

    \medskip\noindent \textit{Subcase 1.2.} $j-i=3$.

    Here we have $(i,j)=(1,4)$. Renaming colors among
    $\{\blue,\green,\purple,\orange\}$ if necessary, we may assume
    \[
        c(x_1)=c(x_4)=\blue,\qquad c(x_2)=\green,\qquad c(x_3)=\purple.
    \]
    By~\eqref{eq:no-red}, no vertex in $Y_1\cup Y_4$ receives the color $\red$.

    \medskip
    \textit{Step 1: behavior of $Y_4$ and $Y_3$.}
    The vertices of $Y_4$ are adjacent to $x_4$, so they cannot receive color
    $\blue$; together with~\eqref{eq:no-red}, this implies that $Y_4$ can only
    use the three colors $\{\green,\purple,\orange\}$. As $Y_4$ induces a $P_4$
    and $c$ is a star coloring, this $P_4$ cannot be bicolored, so all three
    colors $\green,\purple,\orange$ appear on $Y_4$. In particular, there exists
    $y_{4,q}$ with $c(y_{4,q})=c(x_3)=\purple$.

    Now consider $Y_3$. If some vertex $z\in Y_3$ had color $c(x_4)=\blue$, then
    the path $(z, x_3, x_4, y_{4,q})$ would have colors
    $\blue,\purple,\blue,\purple$ and would therefore be a bicolored $P_4$, a
    contradiction. Thus, no vertex of $Y_3$ receives the color $\blue$, and of
    course none receives the color $c(x_3)=\purple$, since all vertices of $Y_3$
    are adjacent to $x_3$. Hence $Y_3$ can use only the three colors
    $\{\red,\green,\orange\}$. Again, as $Y_3$ induces a $P_4$, all three of
    these colors must appear on $Y_3$, so in particular there exists $y_{3,q}$
    with $c(y_{3,q})=c(x_2)=\green$.

    \medskip
    \textit{Step 2: behavior of $Y_2$.}
    If some vertex $z\in Y_2$ had color $c(x_3)=\purple$, then the path
        $(z, x_2, x_3, y_{3,q})$
    would have colors $\purple,\green,\purple,\green$ and would again be a
    bicolored $P_4$. Therefore no vertex of $Y_2$ receives color
    $c(x_3)=\purple$, and no vertex of $Y_2$ receives color $c(x_2)=\green$
    (because all vertices of $Y_2$ are adjacent to $x_2$). Thus $Y_2$ can only
    use the three colors $\{\red,\blue,\orange\}$. As $Y_2$ induces a $P_4$, all
    three of these colors must appear on $Y_2$. In particular, there exists
    $y_{2,q}$ such that $c(y_{2,q})=c(x_1)=\blue$.

    \medskip
    \textit{Step 3: behavior of $Y_1$.}
    If some vertex $z\in Y_1$ had color $c(x_2)=\green$, then the path
       $( z, x_1, x_2, y_{2,q})$
    would have colors $\green,\blue,\green,\blue$ and would be a bicolored
    $P_4$, a contradiction. Hence no vertex of $Y_1$ receives color
    $c(x_2)=\green$. On the other hand, every vertex of $Y_1$ is adjacent to
    $x_1$, so none can receive color $c(x_1)=\blue$; and by \eqref{eq:no-red}
    none receives color $\red$.

    Consequently, the vertices of $Y_1$ can use at most the two colors
    $\{\purple,\orange\}$. Since $Y_1$ induces a $P_4$, any such coloring of
    $Y_1$ with at most two colors yields a bicolored induced $P_4$ in $H$,
    contradicting again that $c$ is a star coloring.

    \smallskip
    Thus Subcase~1.2 is also impossible, and Case~1 cannot occur.

    \medskip\noindent \textit{Case 2.} The vertices $x_1,x_2,x_3,x_4$ receive
    pairwise distinct colors.

    Since $c(x_0)=\red$, the four colors $c(x_1),c(x_2),c(x_3),c(x_4)$ exhaust
    the remaining four colors $\{\blue,\green,\purple,\orange\}$. The vertices
    $y_{1,4}$ and $y_{4,4}$ are both adjacent to $x_0$, so neither of them can
    receive color $\red$. Thus each of $y_{1,4}$ and $y_{4,4}$ must reuse one of
    the colors $c(x_1),\dots,c(x_4)$. Hence there exist indices
    $i,j\in\{1,2,3,4\}$ such that
    \[
        c(y_{1,4})=c(x_i)
        \qquad\text{and}\qquad
        c(y_{4,4})=c(x_j).
    \]
    As the coloring is proper and $y_{1,4}$ (respectively $y_{4,4}$) is adjacent
    to $x_1$ (respectively $x_4$), we must have $i>1$ and $j<4$.

    Suppose that $i=2$ or $j=3$. By symmetry of $H$, we may assume $i=2$. Then
    the vertices of $Y_2$ cannot use color $c(x_0)$ or $c(x_1)$, as either
    choice together with $y_{1,4},x_0,x_2$ or with $y_{1,4},x_1,x_2$ would
    create a bicolored $P_4$. Moreover they cannot use $c(x_2)$, since all
    vertices in $Y_2$ are adjacent to $x_2$. Thus at most two colors remain
    available on $Y_2$, and since $Y_2$ induces a $P_4$, this would again
    produce a bicolored $P_4$ in $H$, contradicting that $c$ is a star coloring.

    We may therefore assume that $i>2$ and $j<3$, so $j<i$, and
    \begin{equation}
    \label{eq:no-red-second}
        \text{no vertex in $Y_i\cup Y_j$ receives the color $\red$,}
    \end{equation}
    as otherwise such a vertex together with $x_0,x_i,y_{i,4}$ or with
    $x_0,x_j,y_{4,4}$ would induce a bicolored $P_4$ in $H$.

    We now mimic the propagation argument from Case~1, starting at $Y_i$ and
    moving towards $Y_j$. First, the vertices of $Y_i$ cannot use colors $\red$
    or $c(x_i)$ (by \eqref{eq:no-red-second} and adjacency to $x_i$), so they
    must use the other three colors, each at least once (otherwise the induced
    $P_4$ on $Y_i$ would be bicolored). In particular, some vertex $y_{i,r}\in
    Y_i$ satisfies $c(y_{i,r})=c(x_{i-1})$. This implies that no vertex in
    $Y_{i-1}$ can receive color $c(x_i)$ (otherwise we would obtain a bicolored
    $P_4$ of the form $(y_{i-1,p},x_{i-1},x_i,y_{i,r})$). Furthermore, none of
    the vertices in $Y_{i-1}$ can use color $c(x_{i-1})$, since they are all
    adjacent to $x_{i-1}$. Thus $Y_{i-1}$ must use exactly the three colors in
    \[
        \{\red,\blue,\green,\purple,\orange\}
        \setminus\{c(x_{i-1}),c(x_i)\},
    \]
    and, as $Y_{i-1}$ induces a $P_4$, these three colors appear at least once.

    If $j=i-1$, then in addition the vertices of $Y_j=Y_{i-1}$ cannot use color
    $\red$ by~\eqref{eq:no-red-second}, so at most two colors remain available
    on $Y_j$. Since $Y_j$ induces a $P_4$, this yields a bicolored induced
    $P_4$, a contradiction. Hence we may assume $j<i-1$. Repeating the above
    argument at most twice, we eventually reach $Y_j$ and obtain that no vertex
    in $Y_j$ can receive color $c(x_{j+1})$, and of course no vertex in $Y_j$
    can use color $c(x_j)$, since all vertices in $Y_j$ are adjacent to $x_j$.

    Combining this with~\eqref{eq:no-red-second}, we see that the vertices in
    $Y_j$ cannot use the three colors
    \[
        \red,\;c(x_{j-1}),\;c(x_j),
    \]
    so at most two colors remain available on $Y_j$. As $Y_j$ induces a $P_4$,
    any such coloring yields a bicolored induced $P_4$ in $H$, again
    contradicting that $c$ is a star coloring.

    In both cases we arrive at a contradiction, so $H$ admits no star
    $5$-coloring. Therefore $\x(H)\ge 6$, and together with the upper bound this
    shows that $\x(H)=6$, as claimed. 

\medskip
    Finally, we show that $H$ is minimal: every proper subgraph $H'$ of $H$
    admits a star $5$-coloring. We may assume that $H'$ is maximal, so $H'$ is
    obtained from $H$ by deleting a single vertex, say $z$.

    If $z\in\{x_0,x_1,x_2,x_3,x_4\}$, then, by the symmetry of $H$, it suffices
    to consider $z\in\{x_0,x_1,x_2\}$. The corresponding star $5$-colorings of
    $H'$ are displayed in the top row of Figure~\ref{fig:2-tree-minimality}.

    It remains to deal with the case $z=y_{i,r}$ for some $i,r\in[4]$. Again, by
    the symmetry of $H$, we may assume $i\in\{1,2\}$. The two graphs in the
    bottom row of Figure~\ref{fig:2-tree-minimality} show partial star
    $5$-colorings of $H-Y_1$ and $H-Y_2$, respectively. We now extend these
    colorings to the remaining vertices in $Y_i\setminus\{z\}$ as follows.
    \begin{itemize}
        \item If $i=1$, we color the vertices of $Y_1\setminus\{z\}$ with the
            colors $\purple$ and $\orange$, alternating along the natural order
            $y_{1,1},y_{1,2},y_{1,3},y_{1,4}$ (skipping $z$), and choosing the
            starting color as $\orange$. In particular, if $z\neq y_{1,4}$ then
            $c(y_{1,4})=\orange$. This alternation guarantees that adjacent
            vertices in $Y_1\setminus\{z\}$ receive distinct colors, so the
            extension is still proper and does not create any new bicolored
            $P_4$, and hence preserves the star property on $H'$. 
        
        \item If $i=2$, we color the vertices of $Y_2\setminus\{z\}$ with the
            colors $\red$ and $\orange$, alternating along the natural order
            $y_{2,1},y_{2,2},y_{2,3},y_{2,4}$ (skipping $z$); the choice of the
            starting color is irrelevant. This alternation guarantees that
            adjacent vertices in $Y_2\setminus\{z\}$ receive distinct colors,
            and since neither $\red$ nor $\orange$ is used on $x_2$ in the
            partial coloring, the extension is still proper and does not create
            any new bicolored $P_4$, so the star property is preserved on $H'$.
    \end{itemize}

    In each case the resulting coloring of $H'$ is proper, and a direct
    verification shows that no induced $P_4$ of $H'$ is bicolored. Hence $H'$
    is star $5$-colorable, thus implying the minimality of $H$. 
\end{proof}

\begin{figure}
    \centering
    \includegraphics[width=1\textwidth]{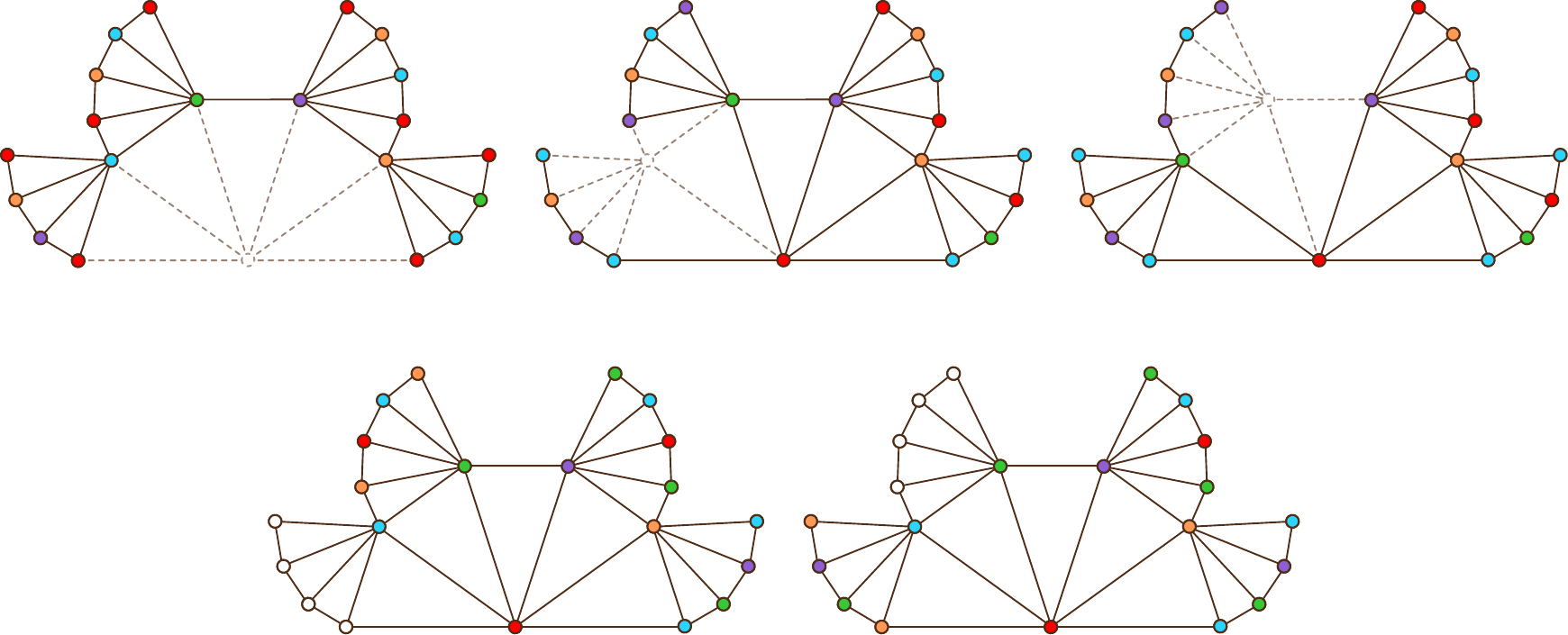}
    \caption{Maximal proper subgraphs of $H$ used in the proof of
    Proposition~\ref{pro:2-tree}. The three graphs in the top row show explicit
    star $5$-colorings of $H-x_0$, $H-x_1$ and $H-x_2$ (up to
    symmetry). The two graphs in the bottom row display partial star
    $5$-colorings of $H-y_{1,r}$ and $H-y_{2,r}$; the remaining
    uncolored vertices in the corresponding fans are then colored with two
    colors as specified in the proof.}
    \label{fig:2-tree-minimality}
\end{figure}

\section*{Acknowledgements}
G.~Benítez-Bobadilla is supported by postdoctoral grant “Estancias Posdoctorales
por México” by Secihti (CVU: 622815).

C.~Hern\'andez-Cruz is supported by UNAM-PAPIIT grant IN106425 and SEP-CONACYT
Ciencia B\'asica A1-S-8397 grant.

C. Linhares Sales is supported by projects number
312044/2022-4 and 404479/2023-5 of CNPq - Conselho Nacional de Desenvolvimento 
Cient\'ifico e Tecnol\'ogico, 
as well as the UNAM-PREI program of the Universidad 
Nacional Aut\'onoma de M\'exico hosted by the 
Facultad de Ciencias. 

A.~Trujillo-Negrete is supported by  Universidad Nacional Autónoma de México
Postdoctoral Program (POSDOC).

\bibliographystyle{plain}
\bibliography{bib}


\appendix

\section{Split minimal obstructions to star 6-colorability}
\label{app:smoS6c}

In this appendix, we present the complete list of split minimal obstructions to
star $6$-coloring that are non-isomorphic to $K_7$, which we obtained with
computer assistance. Each minimal obstruction is encoded in \texttt{graph6}
format by a string~\cite{graph6}. Interested readers may easily obtain drawings
of these graphs from their \texttt{graph6} representation via the \textit{Draw
Graph} feature on the \textit{House of Graphs}
website~\cite{Coolsaet2023HouseOfGraphs}.

\begin{multicols}{3}
\begin{enumerate}

    \item \verb+GJ\z|{+

    \item \verb+GJ\||{+

    \item \verb+GJ]~~{+

    \item \verb+GJn~~{+

    \item \verb+GR\z~{+

    \item \verb+GR~~~{+

    \item \verb+GT\z~{+

    \item \verb+GT\~~{+

    \item \verb+GTlzz{+

    \item \verb+I?CW}]nZw+

    \item \verb+I?CX|\nrw+

    \item \verb+I?C]\\n^w+

    \item \verb+I?C]\\n~w+

    \item \verb+I?Gx{|^rw+

    \item \verb+I?G{y~Nzw+

    \item \verb+I?G{y~N~w+

    \item \verb+I?I]X|^^w+

    \item \verb+I?PHx}nfw+

    \item \verb+I?Qkx|~^w+

    \item \verb+I?Qkz\^Nw+

    \item \verb+I?_zY}~~w+

    \item \verb+I?`Xz^^nw+

    \item \verb+I?`Yx{~Zw+

    \item \verb+I?`hx|^vw+

    \item \verb+I?`ix{~Nw+

    \item \verb+I?`ix{~^w+

    \item \verb+I?`ix{~fw+

    \item \verb+I?`ix{~nw+

    \item \verb+I?`ix{~~w+

    \item \verb+I?`ix|~^w+

    \item \verb+I?`ix|~~w+

    \item \verb+I?`ix}~^w+

    \item \verb+I?`ix~~~w+

    \item \verb+I?aJY{~^w+

    \item \verb+I?aYx|n^w+

    \item \verb+I?bJY{~Nw+

    \item \verb+I?qkz\^Nw+

    \item \verb+J?AGy[ndzN_+

    \item \verb+K??CAKvJw~n^+

    \item \verb+K??CBdfUx^b~+

    \item \verb+K??CISnF|fm^+

    \item \verb+K??E@SfU{nb~+

    \item \verb+L????TPD|Fdnr^+

    \item \verb+L????r@KxZi~b~+

    \item \verb+L???@GXBzFq^r^+

    \item \verb+L???B?Xd[Vg~H~+

    \item \verb+L???C@`VXnd~U~+

    \item \verb+L???CQFJWvg~[^+

    \item \verb+L???C``FYVc~b~+

    \item \verb+L???D?bU]Ja~B~+

    \item \verb+L???D@HTWvc~b~+

    \item \verb+L???E?FKxZqnQ~+

    \item \verb+L???E?rJWzp^P~+

    \item \verb+M?????DAlHjFc~W~_+

    \item \verb+M?????EHWViNeNwn_+

    \item \verb+M????@GGwj`vkNo~_+

    \item \verb+M????A?XirenI~Fn_+

    \item \verb+M????AC@ybdNEnwn_+

    \item \verb+M????AK[HTan_~O~_+

    \item \verb+M????A_ChR_~VNo~_+

    \item \verb+M????B?HWffFQ^`n_+

    \item \verb+N??????CXD`fafPV}Bw+

    \item \verb+N??????_N?ajAvKnhfw+

    \item \verb+N??????o@``V@^Wvrbw+

    \item \verb+O???????@_WBEfHVhL{L^+

    \item \verb+O???????E@ARAVHJoxzo^+

    \item \verb+P??????????R?Zw`rAyiNJH{+

    \item \verb+P?????????W?L`ITbDyG~_x{+

\end{enumerate}
\end{multicols}

\end{document}